\documentclass[11pt]{amsart}

\usepackage[a4paper,margin=1in]{geometry}
\usepackage{amsmath,amssymb,amsfonts,amsthm,amsrefs,bbm}
\usepackage{mathrsfs}
\usepackage{hyperref}
\hypersetup{colorlinks=true,linkcolor=blue,citecolor=blue,urlcolor=blue}
\usepackage{enumitem}
\usepackage{mathtools}
\usepackage{tikz}
\mathtoolsset{showonlyrefs}

\newtheorem{theorem}{Theorem}[section]
\newtheorem{lemma}[theorem]{Lemma}
\newtheorem{proposition}[theorem]{Proposition}
\newtheorem{corollary}[theorem]{Corollary}
\theoremstyle{definition}
\newtheorem{definition}[theorem]{Definition}
\theoremstyle{remark}
\newtheorem{remark}[theorem]{Remark}
\newtheorem{observation}[theorem]{Observation}
\newtheorem{conventions}[theorem]{Conventions}

\newtheorem{example}[theorem]{Example}

\newcommand{\R}{\mathbb R}
\newcommand{\N}{\mathbb N}
\newcommand{\Z}{\mathbb Z}
\newcommand{\C}{\mathbb C}
\newcommand{\U}{\mathcal U}

\newcommand{\id}{\mathrm{id}}
\newcommand{\cb}{\mathrm{cb}}
\newcommand{\CP}{\mathrm{CP}}

\newcommand{\CB}{\mathrm{CB}}

\newcommand{\CC}{\mathrm{CC}}

\newcommand{\diam}{\mathrm{diam}}
\newcommand{\norm}[1]{\left\lVert #1\right\rVert}

\newcommand{\VR}{\mathrm{VR}}

\title[Bures--Kuratowski Metrics and Simplicial Complexes for Completely Bounded Maps]{Bures--Kuratowski Metrics and Simplicial Complexes for Completely Bounded Maps}
\author{Remus Floricel}
\address{University of Regina, Department of Mathematics and Statistics, Regina, SK, Canada}
\email{Remus.Floricel@uregina.ca} 
\author{Sarah Plosker}
\address{Brandon University, Department of Mathematics and Computer Science, Brandon, MB, Canada}
\email{ploskers@brandonu.ca} 
\author{Avner Sadikov}
\address{University of Regina, Department of Mathematics and Statistics, Regina, SK, Canada}
\email{Avner.Sadikov@uregina.ca} 
\date{\today}
\thanks {The research of R.F. and S.P. was partially funded by NSERC Discovery Grants. }
\begin{document}

\begin{abstract}

Let $A$ be a unital $C^*$-algebra and $H$ a Hilbert space. The cone
$\CP(A,B(H))$ of completely positive maps carries the Bures metric $\beta$, closely related to the cb-norm.

We introduce a family of Bures--Kuratowski (BK) metrics on $\CB(A,B(H))$ that extend $\beta$ exactly on $\CP(A,B(H))$. The construction combines a Kuratowski embedding of the Bures cone, based at an anchor $\theta\in\CP(A,B(H))$, with a regular-representation Hausdorff coordinate arising from universal regular models. Each BK metric admits an $\ell^p$-wedge decomposition, splitting $\CB(A,B(H))$ into the Bures cone and a non-CP component attached at $\theta$.

We then study Vietoris--Rips and \v{C}ech complexes of BK metric spaces. The wedge formula yields explicit criteria for mixed simplices, a join-type description of the mixed Rips complex, and ball-intersection criteria for mixed \v{C}ech simplices. For finite point clouds, this makes the mixed simplicial geometry computable from the two component metrics and reveals new homological features arising from the interaction between the CP and non-CP sectors.

\end{abstract}

\maketitle

{\small\tableofcontents}

\section{Introduction}

Let $A$ be a unital $C^*$-algebra and $H$ a Hilbert space. We consider
\[
\mathcal X:=\CB(A,B(H)),
\qquad
\mathcal C:=\CP(A,B(H))\subseteq\mathcal X,
\]
the spaces of completely bounded and completely positive maps, respectively; see \cite{Paulsen, Pisier} for a comprehensive treatment of the subject. 

The cone $\mathcal C$ carries a natural metric geometry given by the Bures
distance $\beta$, defined through common Stinespring representations
\cite{KSW}. This metric is quantitatively linked to the cb-norm by the
Kretschmann--Schlingemann--Werner inequalities \cite{KSW}, and therefore provides a
canonical geometric structure on the completely positive cone. In particular,
when $H=\C$, it reduces to the classical Bures distance on positive functionals \cite{Bures}.

A basic question is whether this geometry extends in a meaningful way from
$\mathcal C$ to the full space $\mathcal X$. A direct extension compatible with
the cb-norm is too much to expect globally. Instead, the point of view of this
paper is that the geometry of $\mathcal X$ should be understood relative
to the CP cone: one should preserve the Bures metric exactly on $\mathcal C$,
while organizing the non-CP maps as a separate geometric sector attached to the
cone in a controlled way. At a conceptual level, the BK framework links the theory of completely bounded
maps in operator algebras, through the induced metric geometry, to simplicial
constructions in algebraic topology.

Our construction uses regular representation theory for completely bounded maps.
By a theorem of Bhat--Mallick--Sumesh \cite{BMS2017}, every completely bounded
map admits a factorization through a regular homomorphism. Universal regular
models for the completely contractive ball, introduced in Definition~\ref{def:univ-reg-model}, then allow us to define an intrinsic
Hausdorff-type metric $\delta_{\mathrm{reg}}$ on $\mathcal X$, obtained by
comparing implementers in universal regular models and taking a normalized
supremum over all such models (see Definition~\ref{def:delta-reg-max}). This metric provides the non-CP coordinate in
our construction.

We then define a family of Bures--Kuratowski metrics $\beta^{BK}_{\theta,\lambda,p,\alpha}$
on $\mathcal X$, depending on an anchor $\theta\in\mathcal C$, a scale
parameter $\lambda>0$, an exponent $p\in[1,\infty]$, and a parameter
$\alpha\in(0,1]$; see Definition~\ref{def:BK-family}. These metrics agree exactly with the Bures distance on
$\mathcal C$  (Theorem~\ref{prop:BK-metric}), and combine two pieces of information: a Kuratowski embedding \cite{Kuratowski} of
the Bures cone based at $\theta$, and a collapsed regular coordinate that
measures the non-CP sector. The resulting metric is intentionally asymmetric
with respect to the anchor: non-CP maps can approach the CP cone only at the
distinguished point $\theta$.

A central structural result is that every BK metric admits a canonical
$\ell^p$-wedge decomposition
\[
(\mathcal X,\beta^{BK}_{\theta,\lambda,p,\alpha})
\cong
(\mathcal C,\beta,\theta)\ \vee_p\ (\mathcal Y,d_{\mathrm{reg}},\ast),
\]
where $\mathcal Y$ is obtained from the non-CP locus by adjoining a basepoint; see Theorem~\ref{thm:BK-lp-wedge}.
Thus the space of completely bounded maps splits into two metric pieces: the
classical Bures geometry on the CP cone and a purely non-CP metric space
attached at the anchor. This decomposition clarifies several basic phenomena:
the CP cone is closed and the non-CP region accumulates on the cone only at the
anchor (Proposition~\ref{prop:bk-attachment-closure}).

The paper also studies the resulting family of BK metrics as a geometric bundle
over the Bures cone. We show that changing $p$, $\lambda$, or $\alpha$ does not
change the induced topology, whereas changing the anchor generally does (Proposition~\ref{prop:anchors-nonequiv}). In
fact, the assignment
\[
\theta\longmapsto \beta^{BK}_{\theta,\lambda,p,\alpha}
\]
defines an isometric embedding of $(\mathcal C,\beta)$ into the space of all
metrics on $\mathcal X$ equipped with the uniform distance (Proposition~\ref{prop:isometric-anchor-map}). In this sense, the
Bures cone parametrizes a canonical family of metric structures on the ambient
space of completely bounded maps.

The final section turns from metric geometry to simplicial constructions. Once a
BK space is identified as an $\ell^p$-wedge, it becomes natural to ask how
Vietoris--Rips and \v{C}ech complexes detect the interaction between the CP and
non-CP sectors. For Vietoris--Rips complexes, the wedge structure yields a
complete and explicit description of mixed simplices: they are governed by the
two radial functions measuring distance to the glued basepoints, leading to a
join-type decomposition of the mixed part of $\VR_t(\mathcal X,d_\theta)$
(Theorem~\ref{thm:rips-join}). In the case $p=\infty$, this description becomes
exact on finite point clouds, where the mixed simplices form a genuine simplicial
join of sublevel-set complexes from the two components.

For \v{C}ech complexes, the situation is more delicate. Mixed simplices are no
longer determined purely by pairwise radial bounds; instead, they reduce to
ball-intersection problems within one component subject to a radial constraint
imposed by the other (Proposition~\ref{prop:cech-mixed-criterion}). This leads to
a natural distinction between intrinsic and ambient \v{C}ech complexes in the BK
setting: ambient complexes can detect the glued basepoint and become
contractible at relatively small scales, while intrinsic complexes retain the
combinatorial geometry of the point cloud. In particular, for finite BK clouds,
the ambient \v{C}ech complex exhibits a ``cone effect'' driven by the basepoint,
which can collapse homology earlier than in the corresponding Vietoris--Rips
filtration. These results show that the wedge geometry produces new and
explicitly computable interactions between the CP and non-CP sectors at the
level of simplicial topology.

\section{Background: Bures distance and  universal regular models}

Let $A$ be a unital $C^*$-algebra and $H$ a Hilbert space.
\begin{definition}
[Bures distance on $\CP(A,B(H))$ {\cite[Def.\ 1]{KSW}}]
\label{def:bures}
Let $\phi,\psi\in \CP(A,B(H))$.
For a representation $\pi:A\to B(K)$, denote by $\mathscr S_\pi(\phi)$ the set of all operators $V:H\to K$
such that $(\pi,V,K)$ is a Stinespring representation of $\phi$, i.e.\ $\phi(a)=V^*\pi(a)V$.
Define
\[
\beta_\pi(\phi,\psi):=\inf\{\|V-W\|:\ V\in \mathscr S_\pi(\phi),\ W\in\mathscr S_\pi(\psi)\},
\]
(with the convention $\beta_\pi(\phi,\psi)=\infty$ if one of the sets is empty), and define the Bures distance by
\[
\beta(\phi,\psi):=\inf_{\pi}\beta_\pi(\phi,\psi),
\]
the infimum taken over all representations $\pi$ of $A$.
\end{definition}

It is well known that $(\mathcal C,\beta)$ is a metric space, and that $\beta$  coincides with the classical Bures distance when 
$H=\mathbb{C}$. Moreover, $\beta$ satisfies the KSW comparison inequalities (see \cite[Th.\ 1]{KSW}), namely 
\begin{equation}\label{eq-KSW}
\frac{\|\phi-\psi\|_{cb}}{\|\phi\|_{cb}^{1/2}+\|\psi\|_{cb}^{1/2}}
\ \le\
\beta (\phi,\psi)
\ \le\
\|\phi-\psi\|_{cb}^{1/2}, \qquad (\phi,\,\psi\in \mathcal C).
\end{equation}
Thus $\beta$ and the cb-norm induce the same topology on  $\CP(A,B(H))$, and are quantitatively comparable via \eqref{eq-KSW}.

\begin{definition}[cf. {\cite[Def.\ 2.7]{BMS2017}}]
Let $A$ and $B$ be unital $C^*$-algebras.
A (multiplicative) homomorphism $\tau:A\to B$ is called regular if
\[
\tau(u)^*\tau(u)=\tau(1)^*\tau(1)
\quad\text{and}\quad
\tau(u)\tau(u)^*=\tau(1)\tau(1)^*
\qquad \forall\, u\in \U(A).
\]
\end{definition}

\begin{theorem}[Regular representation theorem {\cite[Thm.\ 3.2]{BMS2017}}]\label{thm:BMS-3-2}
For every $\phi\in \CB(A,B(H))$ there exist a Hilbert space $K$, a regular homomorphism
$\tau:A\to B(K)$, and an operator $W\in B(H,K)$ such that
\[
\phi(a)=W^*\tau(a)W\qquad\forall a\in A.
\]
\end{theorem}

\begin{remark}
\label{rem:fixed-regular-representation-false}
A completely positive map may admit regular representations with
non-$*$-homomorphic $\tau$. Indeed, let $\phi\in\CP(A,B(H))$ be nonzero, and choose a Stinespring
representation $\phi(a)=V^*\pi(a)V$
with $\pi:A\to B(K)$ a $*$-homomorphism.
Set $\widetilde K:=K\oplus K$ and define
\[
\widetilde\tau(a):=
\begin{bmatrix}
\pi(a) & \pi(a)\\
0 & 0
\end{bmatrix},
\qquad
\widetilde W:=\frac{1}{\sqrt2}\binom{V}{V}.
\]
Then $\widetilde\tau$ is a regular homomorphism, $\phi(a)=\widetilde W^*\,\widetilde\tau(a)\,\widetilde W$ $(a\in A),$
but $\widetilde\tau$ is not a $*$-homomorphism since
$\widetilde\tau(1)
\neq
\widetilde\tau(1)^*.$
\end{remark}

\begin{definition}
\label{def:univ-reg-model}  Let $\mathcal X_1:=\CC(A,B(H))$ be the space of completely contractive linear maps from $A$ to $B(H)$,
A pair $(K,\tau)$, consisting of a Hilbert space $K$ and a regular homomorphism
$\tau:A\to B(K)$, is called a universal regular model for $\CC(A,B(H))$ if for every
$\psi\in \CC(A,B(H))$ there exists an isometry $V_\psi\in B(H,K)$ such that
\[
\psi(a)=V_\psi^*\tau(a)V_\psi\qquad(a\in A).
\]\end{definition}
Let $\mathfrak U$ denote the collection of all universal regular models $(K,\tau)$ for $\CC(A,B(H))$.

\begin{theorem}[{\cite[Thm.\ 3.5]{BMS2017}}]\label{thm:BMS-3-5}
 There exists a universal regular model for $\CC(A,B(H))$.
\end{theorem}

\begin{remark}\label{rem:non-uniqueness}

Universal regular models for $\CC(A,B(H))$ are not unique up to unitary equivalence,
even under a direct analogue of the minimal Stinespring cyclicity condition. 
One can call $(K,\tau)$ minimal if $K=\overline{\mathrm{span}}\{\tau(a)Vh,\tau(a)^*Vh:\ a\in A,\ h\in H,\ V\in\mathcal V\},$
where $\mathcal V$ consists of all isometries $V\in B(H,K)$ implementing maps in $\CC(A,B(H))$.

Non-uniqueness already appears for $(A,H)=(\mathbb C,\mathbb C)$.
For $a\ge3$ define $\tau_a(\lambda)=\lambda P_a$ on $K=\mathbb C^2$, where
\[
P_a=\begin{pmatrix}1&a\\0&0\end{pmatrix}.
\]
Then $(\mathbb C^2,\tau_a)$ is a minimal universal regular model for $\CC(\C,\C)$, but $(\mathbb C^2,\tau_a)$ and $(\mathbb C^2,\tau_b)$ are not unitarily equivalent for $a\neq b$ since
\[
\sigma(P_a^*P_a)=\{0,1+a^2\}\neq\{0,1+b^2\}.
\]
Thus minimality alone does not yield unitary uniqueness of universal regular models.
\end{remark}

\section{Bures--Kuratowski metrics}\label{sec:reg}

Let $A$ be a unital $C^*$-algebra and $H$ a Hilbert space. For each $(K,\tau)\in\mathfrak U$ set $T_\tau:=\tau(1)\in B(K)$. For $\phi\in \mathcal X$, $\phi\neq 0$, define
\[
\mathscr R_\tau(\phi):=
\left\{\|\phi\|_{cb}^{1/2}\,V:\ V\in B(H,K)\text{ is an isometry and }
\widehat \phi(a)=V^*\tau(a)V\ \forall a\in A\right\},\] where $\widehat{\phi}:=\phi/\|\phi\|_{cb} \in  \mathcal X_1$. We also set $\mathscr R_\tau(\phi):=\{0\}$ and $\widehat{\phi}:=0$ if $\phi=0$. We then consider the linear map $\Gamma_\tau:B(H,K)\to B\bigl(H,K\oplus K\bigr)$ defined by
\[
\Gamma_\tau(W):=\frac1{\sqrt2}\begin{bmatrix}W\\ T_\tau W\end{bmatrix},
\] and set $\mathscr F_\tau(\phi):=\Gamma_\tau(\mathscr R_\tau(\phi))$. By Theorem~\ref{thm:BMS-3-5}, the sets $\mathscr R_\tau(\phi)$ and $\mathscr F_\tau(\phi)$ are nonempty.

We begin with basic norm-topological properties of the model-dependent implementer sets.
\begin{lemma}\label{lem:closed-bounded-tau}
Fix $(K,\tau)\in\mathfrak U$ and $\phi\in \mathcal X$. Then the set $\mathscr F_\tau(\phi)$ is norm-bounded, and norm-closed in
$B\bigl(H,K\oplus K\bigr)$.
\end{lemma}

\begin{proof}
Boundedness holds because every $W\in\mathscr R_\tau(\phi)$ has $\|W\|=\sqrt{\|\phi\|_{cb}}$ and
\[
\|\Gamma_\tau(W)\|\le \frac1{\sqrt2}\bigl(\|W\|+\|T_\tau W\|\bigr)\le \frac1{\sqrt2}(1+\|T_\tau\|)\sqrt{\|\phi\|_{cb}}.
\]
For closedness,
assume $\phi\neq 0$.
Let $\Gamma_\tau(W_n)\in\mathscr F_\tau(\phi)$ and assume $\Gamma_\tau(W_n)\to X$ in operator norm.
Reading the first coordinate shows $W_n\to W$ in operator norm, hence $X=\Gamma_\tau(W)$.
Write $W_n=\sqrt{\|\phi\|_{cb}}V_n$ with $V_n$ isometries implementing $\widehat\phi$.
Then $V_n\to V:=W/\sqrt{\|\phi\|_{cb}}$ in norm, so $V^*V=\lim V_n^*V_n=I$ and
$V^*\tau(a)V=\lim V_n^*\tau(a)V_n=\widehat\phi(a)$. Thus $W\in \mathscr R_\tau(\phi)$ and $X\in\mathscr F_\tau(\phi)$.
\end{proof}

\begin{definition}
\label{def:delta-reg-tau}
Fix $(K,\tau)\in\mathfrak U$.
Let $d_H$ denote the Hausdorff distance on the collection of nonempty closed bounded subsets of
the normed space $B\bigl(H,K\oplus K\bigr)$.
Define
\[
\delta_{\mathrm{reg}}^{\tau}(\phi,\psi):=d_H\bigl(\mathscr F_\tau(\phi),\mathscr F_\tau(\psi)\bigr),
\qquad \phi,\psi\in \mathcal X.
\]
\end{definition}
Each fixed-model Hausdorff coordinate induces a metric on $\mathcal X$.
\begin{lemma}
\label{lem:delta-reg-tau-metric}
For each $(K,\tau)\in\mathfrak U$, the function $\delta_{\mathrm{reg}}^{\tau}$ is a metric on $\mathcal X$.
\end{lemma}

\begin{proof}
Symmetry and the triangle inequality follow from the corresponding properties of the Hausdorff distance.
For separation, suppose $\delta_{\mathrm{reg}}^{\tau}(\phi,\psi)=0$. Then $\mathscr F_\tau(\phi)=\mathscr F_\tau(\psi)$.
Pick $W\in\mathscr R_\tau(\phi)$, so $\Gamma_\tau(W)\in\mathscr F_\tau(\phi)=\mathscr F_\tau(\psi)$.
Thus $\Gamma_\tau(W)=\Gamma_\tau(W')$ for some $W'\in\mathscr R_\tau(\psi)$.
Since $\Gamma_\tau$ remembers $W$ in its first coordinate, it is injective, so $W=W'$.
Finally, $\phi(a)=W^*\tau(a)W=\psi(a)$ for all $a\in A$, hence $\phi=\psi$.
\end{proof}
The following example shows that a fixed universal model need not recover the Bures metric on the CP cone.
\begin{example}
\label{ex:fixed-model-mismatch}
In general, for a fixed universal regular model $(K,\tau)$, the coordinate $\delta_{\mathrm{reg}}^{\tau}$ need not agree with $\beta$ on $\mathcal C$.
Take $(A,H)=(\C,\C)$ and fix $a\ge3$. Let $(\C^2,\tau_a)$ be the universal regular model from
Remark~\ref{rem:non-uniqueness}, and for $c\ge0$ define $\phi_c(\lambda)=c\lambda$.

For $c_1,c_2\ge0$,
\begin{equation}\label{eq:1-dim}
\beta(\phi_{c_1},\phi_{c_2})=|\sqrt{c_1}-\sqrt{c_2}|,
\end{equation}
since $\phi_c(\lambda)=V_c^*\lambda V_c$ with $V_c=\sqrt c$, and the reverse inequality follows from
$\|V\|^2=\phi_c(1)=c$ for any Stinespring implementer $V$.

Let $S:=\mathscr F_{\tau_a}(\phi_1)$. Since $\mathscr R_{\tau_a}(\phi_c)=\sqrt c\,\mathscr R_{\tau_a}(\phi_1)$ and $\Gamma_{\tau_a}$ is linear,
\[
\delta_{\mathrm{reg}}^{\tau_a}(\phi_{c_1},\phi_{c_2})
=
|\sqrt{c_1}-\sqrt{c_2}|\sup_{x\in S}\|x\|.
\]
For $v_0=\frac{1}{\sqrt{1+a^2}}\binom{1}{a},$
we have $v_0^*P_a v_0=1$, so $\Gamma_{\tau_a}(v_0)\in S$, and $\|\Gamma_{\tau_a}(v_0)\|^2
=
\frac12\bigl(\|v_0\|^2+\|P_a v_0\|^2\bigr)
=
1+\frac{a^2}{2}.$
Hence
\[
\delta_{\mathrm{reg}}^{\tau_a}(\phi_{c_1},\phi_{c_2})
\ge
\sqrt{1+\frac{a^2}{2}}\,
|\sqrt{c_1}-\sqrt{c_2}|.
\]
Thus $\delta_{\mathrm{reg}}^{\tau_a}$ can differ from $\beta$ on $\CP(\C,\C)$.
\end{example}

The model-dependent radius admits quantitative bounds in terms of the cb-norm.
\begin{lemma}
\label{lem:rad-bounds}
Fix $(K,\tau)\in\mathfrak U$ and let 
$c_+(\tau):=(\sqrt{1+\|T_\tau\|^2})/\sqrt2.$
Then for every $\phi\in \mathcal X$,
\[
{\sqrt2}^{-1}\,\sqrt{\|\phi\|_{cb}}
\ \le\ \delta_{\mathrm{reg}}^{\tau}(\phi,0)
\ \le\ c_+(\tau)\,\sqrt{\|\phi\|_{cb}}.
\]
\end{lemma}

\begin{proof}
If $\phi=0$ the claim is trivial. Assume $\phi\neq 0$ and take $W\in\mathscr R_\tau(\phi)$, so $\|W\|=\sqrt{\|\phi\|_{cb}}$.
Then
\[
\|\Gamma_\tau(W)\|^2=\frac12\|W^*W+W^*T^*_\tau T_\tau W\|.
\]
For the lower bound, $\|W^*W+W^*T^*_\tau T_\tau W\|\ge \|W^*W\|=\|W\|^2=\|\phi\|_{cb}$.
For the upper bound, $W^*T^*_\tau T_\tau W\le \|T_\tau \|^2 W^*W$ implies
$\|W^*W+W^*T^*_\tau T_\tau W\|\le (1+\|T_\tau \|^2)\|W^*W\|=(1+\|T_\tau \|^2)\|\phi\|_{cb}$.
Taking suprema over $W\in\mathscr R_\tau(\phi)$ yields the bounds.
\end{proof}

\begin{definition}
\label{def:delta-reg-max}
Define \[
\delta_{\mathrm{reg}}(\phi,\psi)
:=\sup_{(K,\tau)\in\mathfrak U}\ \overline{\delta}_{\mathrm{reg}}^{\tau}(\phi,\psi),
\]
 for all  $\phi,\psi\in\mathcal X,$ where $\overline{\delta}_{\mathrm{reg}}^{\tau}:=\frac{1}{c_+(\tau)}\,\delta_{\mathrm{reg}}^{\tau}$ and $c_+(\tau)$ is as in Lemma~\ref{lem:rad-bounds}.
\end{definition}
The intrinsic regular coordinate defines a metric compatible with cb-scaling.
\begin{lemma}
\label{lem:deltareg-basic}
The function $\delta_{\mathrm{reg}}$ is a metric on $\mathcal X$. Moreover, for all $\phi,\psi\in\mathcal X$, $\delta_{\mathrm{reg}}(\phi,\psi)\le \|\phi\|_{cb}^{1/2}+\|\psi\|_{cb}^{1/2},$ and in particular $\rho(\phi):=\delta_{\mathrm{reg}}(\phi,0)\le \|\phi\|_{cb}^{1/2}.$
\end{lemma}

\begin{proof}
Each $\overline{\delta}_{\mathrm{reg}}^{\tau}$ is a metric by Lemma~\ref{lem:delta-reg-tau-metric}.
A pointwise supremum of metrics is again a metric: the triangle inequality follows by taking suprema, and if $\delta_{\mathrm{reg}}(\phi,\psi)=0$ then in particular $\overline{\delta}_{\mathrm{reg}}^{\tau}(\phi,\psi)=0$ for any fixed model $(K,\tau)$, so $\phi=\psi$.

For the bound, fix $(K,\tau)\in\mathfrak U$ and use the Hausdorff triangle inequality:
\[
\delta_{\mathrm{reg}}^{\tau}(\phi,\psi)\le d_H(\mathscr F_\tau(\phi),\{0\})+d_H(\mathscr F_\tau(\psi),\{0\}).
\]
By the upper bound in Lemma~\ref{lem:rad-bounds}, one has
$ \delta_{\mathrm{reg}}^{\tau}(\phi,0)\le c_+(\tau)\|\phi\|_{cb}^{1/2}$ and similarly for $\psi$.
Dividing by $c_+(\tau)$ gives
\[
\overline{\delta}_{\mathrm{reg}}^{\tau}(\phi,\psi)\le \|\phi\|_{cb}^{1/2}+\|\psi\|_{cb}^{1/2}.
\]
Taking the supremum over $(K,\tau)\in\mathfrak U$ yields the claimed bound for $\delta_{\mathrm{reg}}$, and the radius estimate is the special case $\psi=0$.
\end{proof}

This leads to the following Kuratowski-type gluing construction.

\begin{definition}
Fix a CP map $\theta\in \mathcal C$ and define $\kappa_\theta:\mathcal C\to \ell^\infty(\mathcal C)$ by
\[
\kappa_\theta(\phi)(\eta):=\beta(\phi,\eta)-\beta(\theta,\eta)\qquad(\eta\in \mathcal C).
\]
Extend $\kappa_\theta$ to all of $\mathcal X$ by setting $\kappa_\theta(\phi):=0$ whenever $\phi\notin \mathcal C$. This discontinuous extension is intentional and drives the wedge phenomenon.
\end{definition}
The Kuratowski embedding realizes the Bures metric exactly inside $\ell^\infty$.
\begin{lemma}\label{lem:kuratowski-isometry-anchor}
For all $\phi,\psi\in \mathcal C$ one has $\|\kappa_\theta(\phi)-\kappa_\theta(\psi)\|_\infty=\beta(\phi,\psi),$ and 
$\|\kappa_\theta(\phi)\|_\infty=\beta(\phi,\theta).$
\end{lemma}

\begin{proof}
This is the standard Kuratowski embedding: $|\beta(\phi,\eta)-\beta(\psi,\eta)|\le \beta(\phi,\psi)$ for all $\eta$,
with equality at $\eta=\psi$. The second identity is the special case $\psi=\theta$.
\end{proof}
\begin{definition}
Define 
\[
\widetilde\delta_{\mathrm{reg}}(\phi,\psi):=
\begin{cases}
0,& \phi,\psi\in \mathcal C,\\
\rho(\phi),& \phi\notin\mathcal C,\ \psi\in \mathcal C,\\
\rho(\psi),& \phi\in\mathcal C,\ \psi\notin \mathcal C,\\
\delta_{\mathrm{reg}}(\phi,\psi),& \phi,\psi\notin \mathcal C,
\end{cases}
\]
where $\rho(\cdot)$ is as in Lemma~\ref{lem:deltareg-basic}.
\end{definition}
We notice that the function $\widetilde\delta_{\mathrm{reg}}$ satisfies symmetry and the triangle inequality, but it is a pseudometric:
it vanishes on $\mathcal C\times \mathcal C$ by construction.

\begin{definition}[Bures--Kuratowski family]\label{def:BK-family}
Fix parameters $\theta\in \mathcal C,$ $\lambda>0,$ $p\in[1,\infty],$ and $\alpha\in(0,1].$
Define the Bures--Kuratowski metric $\beta^{BK}_{\theta,\lambda,p,\alpha}$ on $\mathcal X$ by
\[
\beta^{BK}_{\theta,\lambda,p,\alpha}(\phi,\psi)
:=
\Bigl\|\bigl(
\|\kappa_\theta(\phi)-\kappa_\theta(\psi)\|_\infty,\ 
\lambda\,(\widetilde\delta_{\mathrm{reg}}(\phi,\psi))^\alpha
\bigr)\Bigr\|_{\ell^p(\R^2)}.
\]
\end{definition}
The BK construction yields a genuine extension of the Bures metric.
\begin{theorem}
\label{prop:BK-metric}
For every choice of parameters $(\theta,\lambda,p,\alpha)$, the function $\beta^{BK}_{\theta,\lambda,p,\alpha}$ is a metric on $\mathcal X$.
Moreover, for $\phi,\psi\in \mathcal C$ one has $\beta^{BK}_{\theta,\lambda,p,\alpha}(\phi,\psi)=\beta(\phi,\psi).$
\end{theorem}

\begin{proof}
Both coordinates $d_1(\phi,\psi):=\|\kappa_\theta(\phi)-\kappa_\theta(\psi)\|_\infty$ and
$d_2(\phi,\psi):=(\widetilde\delta_{\mathrm{reg}}(\phi,\psi))^\alpha$
satisfy the triangle inequality (for $d_2$ this uses that $\widetilde\delta_{\mathrm{reg}}$ is a pseudometric and $t\mapsto t^\alpha$ is subadditive for $\alpha\in(0,1]$).
Hence any $\ell^p$-combination satisfies the triangle inequality.

If $\beta^{BK}_{\theta,\lambda,p,\alpha}(\phi,\psi)=0$ then $d_1(\phi,\psi)=0$ and $d_2(\phi,\psi)=0$.
If $\phi,\psi\in \mathcal C$, then $d_1(\phi,\psi)=\beta(\phi,\psi)$ by Lemma~\ref{lem:kuratowski-isometry-anchor}, hence $\phi=\psi$.
If $\phi,\psi\notin \mathcal C$, then $d_2(\phi,\psi)=\delta_{\mathrm{reg}}(\phi,\psi)^\alpha=0$ forces $\delta_{\mathrm{reg}}(\phi,\psi)=0$, hence $\phi=\psi$.
If exactly one of $\phi,\psi$ lies in $\mathcal C$, then $d_2(\phi,\psi)=\rho(\text{non-CP map})^\alpha>0$ (since $\rho(\eta)=0$ iff $\eta=0\in\mathcal C$), a contradiction.
Thus $\phi=\psi$, and $\beta^{BK}_{\theta,\lambda,p,\alpha}$ is a metric.

Finally, if $\phi,\psi\in \mathcal C$ then $\widetilde\delta_{\mathrm{reg}}(\phi,\psi)=0$ and
$\|\kappa_\theta(\phi)-\kappa_\theta(\psi)\|_\infty=\beta(\phi,\psi)$, proving CP-reduction.
\end{proof}

The next two propositions show that no BK metric admits a global KSW-type upper bound on $\mathcal X$, while a KSW-type lower bound holds on the non-CP sector.

\begin{proposition}
\label{prop:no-upper-KSW}
Fix $\theta\in \mathcal C$, $\lambda>0$, $p\in[1,\infty]$, and $\alpha\in(0,1]$.
Then there does not exist a function $f:[0,\infty)\to[0,\infty)$ with $f(t)\to 0$ as $t\downarrow 0$
 such that
\[
\beta^{BK}_{\theta,\lambda,p,\alpha}(\phi,\psi)\le f\bigl(\|\phi-\psi\|_{cb}\bigr)
\qquad\forall\, \phi,\psi\in \mathcal X.
\]
In particular, there is no constant $C>0$ such that $\beta^{BK}_{\theta,\lambda,p,\alpha}(\phi,\psi)\le C\,\|\phi-\psi\|_{cb}^{1/2}$
for all $\phi,\psi\in \mathcal X$.
\end{proposition}

\begin{proof}
Let $\omega$ be a state on $A$, and define the unital completely positive map $\eta(a):=\omega(a)I_H$, $a\in A$.
Choose $\phi:=0$ if $ \theta\neq 0$, and  $\phi:=\eta$ if $\theta=0$. Then $\phi\in \mathcal C$ and $\phi\neq \theta$.
For $n\in\mathbb N$, define
\[
\psi_n(a):=\phi(a)+\frac{i}{n}\,\eta(a)\qquad(a\in A).
\]
Since $\phi(1)$ is selfadjoint and $\psi_n(1)$ is not, one has $\psi_n\notin \mathcal C$ for every $n$.
Moreover, $\|\psi_n-\phi\|_{cb}=\frac{1}{n}\,\|\eta\|_{cb}=\frac{1}{n},$
because $\eta$ is unital completely positive.

On the other hand, $\kappa_\theta(\psi_n)=0$ by definition, while $\phi\in\mathcal C$ gives $\|\kappa_\theta(\phi)\|_\infty=\beta(\phi,\theta)$
by Lemma~\ref{lem:kuratowski-isometry-anchor}. Therefore
\[
\beta^{BK}_{\theta,\lambda,p,\alpha}(\phi,\psi_n)
=\Bigl\|\bigl(\|\kappa_\theta(\phi)\|_\infty,\,\lambda\rho(\psi_n)^\alpha\bigr)\Bigr\|_{\ell^p(\mathbb R^2)}
\ge \|\kappa_\theta(\phi)\|_\infty
=\beta(\phi,\theta)>0.
\]
Hence any function $f$ as above would satisfy
\[
f\Bigl(\frac{1}{n}\Bigr)\ge \beta^{BK}_{\theta,\lambda,p,\alpha}(\phi,\psi_n)\ge \beta(\phi,\theta)>0
\qquad\forall n\in\mathbb N,
\]
contradicting $f(t)\to 0$ as $t\downarrow 0$.
The final statement follows by taking $f(t)=Ct^{1/2}$.
\end{proof}

\begin{proposition}
\label{prop:lower-KSW-nonCP}
Fix $(K,\tau)\in\mathfrak U$.
Then for all $\phi,\psi\in \mathcal X$,
\[
\|\phi-\psi\|_{cb}
\le
\sqrt2\,\|\tau\|_{cb}\, c_+(\tau)\,\bigl(\|\phi\|_{cb}^{1/2}+\|\psi\|_{cb}^{1/2}\bigr)\,\delta_{\mathrm{reg}}(\phi,\psi).
\]
Consequently, for every $\theta\in \mathcal C$, $\lambda>0$, $p\in[1,\infty]$, $\alpha\in(0,1]$, and all
$\phi,\psi\in\mathcal X\setminus \mathcal C$,
\[
\beta^{BK}_{\theta,\lambda,p,\alpha}(\phi,\psi)
\ge
\lambda\left(
\frac{\|\phi-\psi\|_{cb}}
{\sqrt2\,\|\tau\|_{cb}\, c_+(\tau)\,\bigl(\|\phi\|_{cb}^{1/2}+\|\psi\|_{cb}^{1/2}\bigr)}
\right)^\alpha.
\]
\end{proposition}

\begin{proof}
Fix $\varepsilon>0$. By definition of the Hausdorff distance $\delta_{\mathrm{reg}}^{\tau}$,
for any $X\in\mathscr F_{\tau}(\phi)$ there exists $Y\in\mathscr F_{\tau}(\psi)$ such that
\[
\|X-Y\|\le \delta_{\mathrm{reg}}^{\tau}(\phi,\psi)+\varepsilon.
\]
Choose $W\in\mathscr R_{\tau}(\phi)$ with $X=\Gamma_{\tau}(W)$, and pick $W'\in\mathscr R_{\tau}(\psi)$ with
$Y=\Gamma_{\tau}(W')$ satisfying the bound. Since $\Gamma_{\tau}$ remembers $W$ in its first coordinate,
\[
\|W-W'\|\le \sqrt2\,\|\Gamma_{\tau}(W)-\Gamma_{\tau}(W')\|
\le \sqrt2\bigl(\delta_{\mathrm{reg}}^{\tau}(\phi,\psi)+\varepsilon\bigr).
\]
Now $\phi(a)=W^*\tau(a)W$ and $\psi(a)=W'^*\tau(a)W'$ for all $a\in A$, hence for any $n$ and any $x\in M_n(A)$,
\[
\|\phi^{(n)}(x)-\psi^{(n)}(x)\|
=
\|(W-W')^*\tau^{(n)}(x)W + W'^*\tau^{(n)}(x)(W-W')\|
\le \|\tau\|_{cb}\,\|x\|\,(\|W\|+\|W'\|)\,\|W-W'\|.
\]
Taking suprema over $n$ and $x$ yields
\[
\|\phi-\psi\|_{cb}
\le \|\tau\|_{cb}\,(\|W\|+\|W'\|)\,\|W-W'\|.
\]
By definition of $\mathscr R_{\tau}(\cdot)$, $\|W\|=\|\phi\|_{cb}^{1/2}$ and $\|W'\|=\|\psi\|_{cb}^{1/2}$.
Letting $\varepsilon\downarrow 0$ gives
\[
\|\phi-\psi\|_{cb}
\le
\sqrt2\,\|\tau\|_{cb}\,\bigl(\|\phi\|_{cb}^{1/2}+\|\psi\|_{cb}^{1/2}\bigr)\,\delta_{\mathrm{reg}}^{\tau}(\phi,\psi).
\]
Since
\[
\delta_{\mathrm{reg}}(\phi,\psi)
=\sup_{(L,\sigma)\in\mathfrak U}\frac{1}{c_+(\sigma)}\,\delta_{\mathrm{reg}}^{\sigma}(\phi,\psi)
\ge \frac{1}{c_+(\tau)}\,\delta_{\mathrm{reg}}^{\tau}(\phi,\psi),
\]
it follows that $\delta_{\mathrm{reg}}^{\tau}(\phi,\psi)\le c_+(\tau)\,\delta_{\mathrm{reg}}(\phi,\psi),$
and therefore
\[
\|\phi-\psi\|_{cb}
\le
\sqrt2\,\|\tau\|_{cb}\, c_+(\tau)\,\bigl(\|\phi\|_{cb}^{1/2}+\|\psi\|_{cb}^{1/2}\bigr)\,\delta_{\mathrm{reg}}(\phi,\psi).
\]
This proves the first inequality.

If now $\phi,\psi\notin\mathcal C$, then by definition of the BK metric,
\[
\beta^{BK}_{\theta,\lambda,p,\alpha}(\phi,\psi)
=\bigl\|(0,\lambda\,\delta_{\mathrm{reg}}(\phi,\psi)^\alpha)\bigr\|_{\ell^p(\mathbb R^2)}
=\lambda\,\delta_{\mathrm{reg}}(\phi,\psi)^\alpha.
\]
Hence
\[
\delta_{\mathrm{reg}}(\phi,\psi)=\lambda^{-1/\alpha}\bigl(\beta^{BK}_{\theta,\lambda,p,\alpha}(\phi,\psi)\bigr)^{1/\alpha},
\]
and substituting this into the first inequality yields the stated estimate on
$\mathcal X\setminus\mathcal C$.
\end{proof}

\section{The Bures--Kuratowski metric bundle}

With notation as in the previous section, we first show that the BK topology is independent of the choice of $\ell^p$-norm and scaling parameter.
\begin{proposition}
\label{prop:p-lambda-equivalent}
Fix $\theta\in \mathcal C$ and $\alpha\in(0,1]$. For any $p,q\in[1,\infty]$ and $\lambda,\mu>0$,
the metrics $\beta^{BK}_{\theta,\lambda,p,\alpha}$ and $\beta^{BK}_{\theta,\mu,q,\alpha}$ are bi-Lipschitz equivalent.
In particular, they induce the same topology on $\mathcal X$.
\end{proposition}
\begin{proof}
Since all norms on $\R^2$ are equivalent, there exists a constant $C_{p,q}\ge 1$ such that
\[
C_{p,q}^{-1}\|(a,b)\|_{\ell^p}\le \|(a,b)\|_{\ell^q}\le C_{p,q}\|(a,b)\|_{\ell^p}
\qquad\text{for all }(a,b)\in\R^2.
\]
Also, for every $\lambda,\mu>0$,
\[
\min\{1,\mu/\lambda\}\,\|(a,\lambda b)\|_{\ell^q}
\le
\|(a,\mu b)\|_{\ell^q}
\le
\max\{1,\mu/\lambda\}\,\|(a,\lambda b)\|_{\ell^q}
\]
for all $(a,b)\in\R^2$. Now apply these inequalities with $a=\|\kappa_\theta(\phi)-\kappa_\theta(\psi)\|_\infty,$
$b=\widetilde\delta_{\mathrm{reg}}(\phi,\psi)^\alpha.$
Then, for all $\phi,\psi\in\mathcal X$,
\[
C_{p,q}^{-1}\min\{1,\mu/\lambda\}\,
\beta^{BK}_{\theta,\lambda,p,\alpha}(\phi,\psi)
\le
\beta^{BK}_{\theta,\mu,q,\alpha}(\phi,\psi)
\le
C_{p,q}\max\{1,\mu/\lambda\}\,
\beta^{BK}_{\theta,\lambda,p,\alpha}(\phi,\psi),
\]
and therefore the two metrics are bi-Lipschitz equivalent.
\end{proof}

The exponent parameter does not affect the induced topology.
\begin{proposition}\label{prop:alpha-topology}
Fix $\theta\in \mathcal C$, $\lambda>0$, and $p\in[1,\infty]$. If $\alpha,\alpha'\in(0,1]$, then
$\beta^{BK}_{\theta,\lambda,p,\alpha}$ and $\beta^{BK}_{\theta,\lambda,p,\alpha'}$ induce the same topology on $\mathcal X$.
\end{proposition}

\begin{proof}
Write $d_\alpha(\phi,\psi):=\beta^{BK}_{\theta,\lambda,p,\alpha}(\phi,\psi),$
$d_{\alpha'}(\phi,\psi):=\beta^{BK}_{\theta,\lambda,p,\alpha'}(\phi,\psi),$ for simplicity, 
and set $u(\phi,\psi):=\|\kappa_\theta(\phi)-\kappa_\theta(\psi)\|_\infty,$ and
$v(\phi,\psi):=\widetilde\delta_{\mathrm{reg}}(\phi,\psi).$
Then $d_\alpha(\phi,\psi)=
\bigl\|\bigl(u(\phi,\psi),\,\lambda\,v(\phi,\psi)^\alpha\bigr)\bigr\|_{\ell^p},$ and similarly $d_{\alpha'}(\phi,\psi)=
\bigl\|\bigl(u(\phi,\psi),\,\lambda\,v(\phi,\psi)^{\alpha'}\bigr)\bigr\|_{\ell^p}.$

It suffices to show that the identity map $\mathrm{id}:(\mathcal X,d_\alpha)\to (\mathcal X,d_{\alpha'})$
is continuous.
Fix $\varepsilon>0$ and define \[\omega_{\alpha,\alpha'}(t):=
\Bigl\|\bigl(t,\lambda^{\,1-\alpha'/\alpha}t^{\alpha'/\alpha}\bigr)\Bigr\|_{\ell^p},\] for every $ t\ge 0.$
Since $t\mapsto t^{\alpha'/\alpha}$ is continuous on $[0,\infty)$ and vanishes at $0$,
we have $\omega_{\alpha,\alpha'}(t)\to 0$
as $t\downarrow 0$.
Choose $\delta>0$ such that $\omega_{\alpha,\alpha'}(\delta)<\varepsilon,$ and let $\phi,\psi\in\mathcal X$ satisfy $d_\alpha(\phi,\psi)<\delta$.
Since each coordinate of a vector is bounded above by its $\ell^p$-norm, we obtain $u(\phi,\psi)<\delta,$ and $\lambda\,v(\phi,\psi)^\alpha<\delta.$
Hence $v(\phi,\psi)<(\delta/\lambda)^{1/\alpha},$
and therefore
\[
\lambda\,v(\phi,\psi)^{\alpha'}
<
\lambda\Bigl(\frac{\delta}{\lambda}\Bigr)^{\alpha'/\alpha}
=
\lambda^{\,1-\alpha'/\alpha}\delta^{\alpha'/\alpha}.
\]
Using that the $\ell^p$-norm is monotone in each coordinate on $[0,\infty)^2$, we get
\[
d_{\alpha'}(\phi,\psi)
=
\Bigl\|\bigl(u(\phi,\psi),\,\lambda\,v(\phi,\psi)^{\alpha'}\bigr)\Bigr\|_{\ell^p}
\le
\Bigl\|\bigl(\delta,\lambda^{\,1-\alpha'/\alpha}\delta^{\alpha'/\alpha}\bigr)\Bigr\|_{\ell^p}
=
\omega_{\alpha,\alpha'}(\delta)
<
\varepsilon.
\]
Thus $B_{d_\alpha}(\phi,\delta)\subseteq B_{d_{\alpha'}}(\phi,\varepsilon)$ for all $\phi\in\mathcal X$,
so $\mathrm{id}:(\mathcal X,d_\alpha)\to(\mathcal X,d_{\alpha'})$ is continuous. By symmetry, the same argument with $\alpha$ and $\alpha'$ interchanged completes the proof.
\end{proof}
We next show that distinct anchors give rise to inequivalent metric topologies.
\begin{proposition}\label{prop:anchors-nonequiv}
Let $\theta_1,\theta_2\in \mathcal C$ with $\theta_1\neq \theta_2$. Fix $\lambda>0$, $p\in[1,\infty]$, and $\alpha\in(0,1]$.
Then the metrics $\beta^{BK}_{\theta_1,\lambda,p,\alpha}$ and $\beta^{BK}_{\theta_2,\lambda,p,\alpha}$ induce different topologies on $\mathcal X$.
\end{proposition}

\begin{proof}
Choose any state $\omega$ on $A$ and define $\psi_n:A\to B(H)$ by
$\psi_n(a):=\frac{i}{n}\,\omega(a)\,I_H$ for all $n\in\N$.
Then $\|\psi_n\|_{cb}=1/n\to 0$, and $\psi_n\notin \mathcal C$ for all $n$ since $\psi_n(1)=\frac{i}{n}I_H$ is not positive.

For the BK-metric defined at $\theta_1$:
$\kappa_{\theta_1}(\theta_1)=0$ and $\kappa_{\theta_1}(\psi_n)=0$, so the Kuratowski term is $0$ and
\[
\widetilde\delta_{\mathrm{reg}}(\theta_1,\psi_n)=\rho(\psi_n)\le \|\psi_n\|_{cb}^{1/2}\to 0
\]
by Lemma~\ref{lem:deltareg-basic}. Hence $\beta^{BK}_{\theta_1,\lambda,p,\alpha}(\theta_1,\psi_n)\to 0$.

For the BK-metric defined at $\theta_2$:
$\kappa_{\theta_2}(\psi_n)=0$ and
\[
\|\kappa_{\theta_2}(\theta_1)-\kappa_{\theta_2}(\psi_n)\|_\infty=\|\kappa_{\theta_2}(\theta_1)\|_\infty=\beta(\theta_1,\theta_2)>0,
\]
so $\beta^{BK}_{\theta_2,\lambda,p,\alpha}(\theta_1,\psi_n)\ge \beta(\theta_1,\theta_2)$ for all $n$.
Thus $\psi_n$ converges to $\theta_1$ in one topology but not the other.
\end{proof}

\begin{proposition}
\label{prop:no-anchor-cb-topology}
Fix $\theta\in \mathcal C$.
Then the metric topology on $\mathcal X$ induced by $\beta^{BK}_{\theta,\lambda,p,\alpha}$ does not coincide with
the cb-norm topology.
\end{proposition}

\begin{proof}
Choose any $\phi\in \mathcal C$ with $\phi\neq \theta$; notice that $\beta(\phi,\theta)>0$. Let $\omega$ be any state on $A$, and define a sequence $\psi_n\in \mathcal X$ as in the proof of Proposition~\ref{prop:no-upper-KSW}, i.e. 
\[
\psi_n(a):=\phi(a)+\frac{i}{n}\,\omega(a)\,I_H\qquad(a\in A).
\]
Then $\|\psi_n-\phi\|_{\cb}=1/n\to 0$, so $\psi_n\to\phi$ in cb-norm.
However, each $\psi_n\notin \mathcal C$ because $\psi_n(1)=\phi(1)+\frac{i}{n}I_H$ is not selfadjoint.

Since $\kappa_\theta(\psi_n)=0$ for all $n$, one has $\|\kappa_\theta(\phi)-\kappa_\theta(\psi_n)\|_\infty=\|\kappa_\theta(\phi)\|_\infty=\beta(\phi,\theta)>0.$
Therefore $\beta^{BK}_{\theta,\lambda,p,\alpha}(\phi,\psi_n)\ge \beta(\phi,\theta)$ for all $n$, so $\psi_n$ does not converge to $\phi$ in the BK topology.
\end{proof}

\begin{definition}
\label{def:D-infty-metrics}
Let $\mathrm{Met}(\mathcal X)$ denote the collection of all metrics on $\mathcal X$.
For $d_1,d_2\in \mathrm{Met}(\mathcal X)$ define the extended pseudo-distance
\[
D_\infty(d_1,d_2):=\sup_{\phi,\psi\in \mathcal X}\bigl|d_1(\phi,\psi)-d_2(\phi,\psi)\bigr|\in[0,\infty].
\]
\end{definition}

The BK construction embeds the CP cone isometrically into the space of metrics.

\begin{proposition}\label{prop:isometric-anchor-map} Fix $\lambda>0$, $p\in[1,\infty]$, and $\alpha\in(0,1]$. Then the map $\Theta:(\mathcal C,\beta)\to(\mathrm{Met}(\mathcal X),D_\infty)$, $\Theta(\theta):=\beta^{BK}_{\theta,\lambda,p,\alpha}$, is an isometric embedding.
\end{proposition}

\begin{proof} For simplicity, denote $\beta^{BK}_\theta:=\beta^{BK}_{\theta,\lambda,p,\alpha},$ for all $\theta\in\mathcal C$. Let $\theta_1,\theta_2\in \mathcal C$ be fixed. We show that
$D_\infty\bigl(\beta^{BK}_{\theta_1},\beta^{BK}_{\theta_2}\bigr)=\beta(\theta_1,\theta_2)$.

(``$\le$'').
It is enough to show that
\begin{equation}\label{eq:lip1}
\bigl|\beta^{BK}_{\theta_1}(\phi,\psi)-\beta^{BK}_{\theta_2}(\phi,\psi)\bigr|
\le \beta(\theta_1,\theta_2),
\qquad (\phi,\psi\in \mathcal X).
\end{equation}
Fix  $\phi$ and $\psi$ in $\mathcal X$.
If $\phi,\psi\in \mathcal C$, then $\widetilde\delta_{\rm reg}(\phi,\psi)=0$ and CP-reduction gives
$\beta^{BK}_{\theta}(\phi,\psi)=\beta(\phi,\psi)$ for every $\theta$, so the difference is $0$.

If $\phi,\psi\notin \mathcal C$, then $\kappa_{\theta}(\phi)=\kappa_{\theta}(\psi)=0$ for every $\theta$, hence
$\beta^{BK}_{\theta}(\phi,\psi)=\lambda(\delta_{\rm reg}(\phi,\psi))^\alpha$ is independent of $\theta$ and the difference is again $0$.

It remains to treat the cross case. By symmetry assume $\phi\in \mathcal C$ and $\psi\notin \mathcal C$.
Then $\kappa_{\theta}(\psi)=0$ and $\|\kappa_\theta(\phi)\|_\infty=\beta(\phi,\theta)$, for all $\theta\in \mathcal C$, by Lemma~\ref{lem:kuratowski-isometry-anchor}.
Also $\widetilde\delta_{\rm reg}(\phi,\psi)=\rho(\psi)$ is independent of $\theta$.
Set $c:=\lambda(\rho(\psi))^\alpha\ge 0$. Then
\[
\beta^{BK}_{\theta}(\phi,\psi)=\|( \beta(\phi,\theta),\,c)\|_{\ell^p}=g_c\bigl(\beta(\phi,\theta)\bigr).
\]
Since the function $[0,\infty)\ni t\mapsto \|(t,c)\|_{\ell^p(\mathbb R^2)} \in[0,\infty)$
is $1$-Lipschitz for any $c\ge 0$, we have
\[
\bigl|\beta^{BK}_{\theta_1}(\phi,\psi)-\beta^{BK}_{\theta_2}(\phi,\psi)\bigr|
\le \bigl|\beta(\phi,\theta_1)-\beta(\phi,\theta_2)\bigr|
\le \beta(\theta_1,\theta_2),
\]
which is \eqref{eq:lip1}.

\medskip
(``$\ge $'')
Let $\varepsilon>0$.
Set $\phi:=\theta_1\in \mathcal C$ and choose $\psi_n\in \mathcal X\setminus \mathcal C$ as in the proof of Proposition~\ref{prop:anchors-nonequiv} so that $\|\psi_n\|_{\cb}\to 0$ as $n\to\infty$.
Then $t_n:=\lambda(\rho(\psi_n))^\alpha\to 0$ by Lemma~\ref{lem:deltareg-basic}.
Compute
\[
\beta^{BK}_{\theta_2}(\phi,\psi_n)-\beta^{BK}_{\theta_1}(\phi,\psi_n)
=\bigl\|(\beta(\theta_1,\theta_2),t_n)\bigr\|_{\ell^p}-t_n
\ \xrightarrow[n\to\infty]{}\ \beta(\theta_1,\theta_2).
\]
Hence for $n$ large enough,
$\beta^{BK}_{\theta_2}(\phi,\psi_n)-\beta^{BK}_{\theta_1}(\phi,\psi_n)\ge \beta(\theta_1,\theta_2)-\varepsilon$,
so $D_\infty(\beta^{BK}_{\theta_1},\beta^{BK}_{\theta_2})\ge \beta(\theta_1,\theta_2)-\varepsilon$.
Let $\varepsilon\downarrow 0$.
\end{proof}
Different anchors produce uniformly close metric structures on compact subsets.
\begin{corollary}
\label{cor:rough-iso-GH}
Let $\theta_1,\theta_2\in \mathcal C$ and set $\varepsilon:=\beta(\theta_1,\theta_2)$.
Then:
\begin{enumerate}[label=\textup{(\roman*)}]
\item For every subset $E\subseteq \mathcal X$, the identity map $\id_E:(E,\beta^{BK}_{\theta_1})\to(E,\beta^{BK}_{\theta_2})$
is a $(1,\varepsilon)$-quasi-isometry, i.e.\ for all $x,y\in E$,
\[
\beta^{BK}_{\theta_1}(x,y)-\varepsilon\ \le\ \beta^{BK}_{\theta_2}(x,y)\ \le\ \beta^{BK}_{\theta_1}(x,y)+\varepsilon.
\]
\item If $E\subseteq \mathcal X$ is compact with respect to both metrics $\beta^{BK}_{\theta_1}$ and $\beta^{BK}_{\theta_2}$,
then the Gromov--Hausdorff distance satisfies
\[
d_{GH}\bigl((E,\beta^{BK}_{\theta_1}),\, (E,\beta^{BK}_{\theta_2})\bigr)\ \le\ \frac{\varepsilon}{2}.
\]
\end{enumerate}
\end{corollary}

\begin{proof}
Write $d_1:=\beta^{BK}_{\theta_1},$ $d_2:=\beta^{BK}_{\theta_2},$ for simplicity.

(i) is \eqref{eq:lip1} restricted to $E\times E$.

(ii)  Assume that $E$ is compact with respect to both $d_1$ and $d_2$. Consider the diagonal correspondence $\Delta_E:=\{(x,x):x\in E\}\subseteq E\times E.$
Its projections onto the two factors are both equal to $E$, so $\Delta_E$ is indeed a correspondence between $(E,d_1)$ and $(E,d_2)$.

Recall that for compact metric spaces $X$ and $Y$, $d_{GH}(X,Y)=\frac12\inf_R \operatorname{dis}(R),$
where the infimum is taken over all correspondences $R\subseteq X\times Y$ and
\[\operatorname{dis}(R):=
\sup\Bigl\{\bigl|d_X(x,x')-d_Y(y,y')\bigr|:(x,y),(x',y')\in R\Bigr\};
\]
see \cite[Theorem~7.3.25]{BBI}. In particular,
\[
d_{GH}(X,Y)\le \frac12\,\operatorname{dis}(R)
\]
for every correspondence $R$. Applying this to $X=(E,d_1)$, $Y=(E,d_2)$, and $R=\Delta_E$, we obtain
\[
d_{GH}\bigl((E,d_1),(E,d_2)\bigr)
\le
\frac12\,\operatorname{dis}(\Delta_E).
\]
Now, if $(x,x),(y,y)\in \Delta_E$, then $\bigl|d_1(x,y)-d_2(x,y)\bigr|\le \varepsilon$
by part (i). Hence $\operatorname{dis}(\Delta_E)
=
\sup_{x,y\in E}\bigl|d_1(x,y)-d_2(x,y)\bigr|
\le \varepsilon.$
Therefore
\[
d_{GH}\bigl((E,\beta^{BK}_{\theta_1}),\,(E,\beta^{BK}_{\theta_2})\bigr)
=
d_{GH}\bigl((E,d_1),(E,d_2)\bigr)
\le
\frac{\varepsilon}{2}.
\]
This proves (ii).
\end{proof}

\section{An \texorpdfstring{$\ell^p$}{lp} wedge decomposition for BK metrics}

Fix $\lambda>0$, $\alpha\in(0,1]$, and $p\in[1,\infty]$. We begin by isolating the non-CP component and equipping it with a natural metric structure.

\begin{definition}[Collapsed non-CP space]\label{def:bk-Y-space}
Let $\mathcal Y:=(\mathcal X\setminus\mathcal C)\ \sqcup\ \{\ast\}$
be the disjoint union of the non-CP region with a distinguished basepoint $\ast$.
Define a metric $d_{\mathrm{reg}}$ on $\mathcal Y$ by
\[
d_{\mathrm{reg}}(x,y):=
\begin{cases}
\lambda\bigl(\delta_{\mathrm{reg}}(x,y)\bigr)^\alpha, & x,y\in \mathcal X\setminus\mathcal C,\\[1mm]
\lambda\bigl(\rho(x)\bigr)^\alpha, & x\in \mathcal X\setminus\mathcal C,\ y=\ast,\\[1mm]
\lambda\bigl(\rho(y)\bigr)^\alpha, & y\in \mathcal X\setminus\mathcal C,\ x=\ast,\\[1mm]
0,& x=y=\ast
\end{cases}
\]
where $\rho(\psi):=\delta_{\mathrm{reg}}(\psi,0) \in (0,\infty)$ is as in Lemma~\ref{lem:deltareg-basic}.
\end{definition}
The following lemma shows that the collapsed non-CP component is indeed a metric space.
\begin{lemma}\label{lem:bk-dreg-metric}
The function $d_{\mathrm{reg}}$ is a metric on $\mathcal Y$.
\end{lemma}

\begin{proof}
Symmetry and nonnegativity are clear.
If $x,y,z\in\mathcal X\setminus\mathcal C$, then $\delta_{\mathrm{reg}}$ is a metric and $t\mapsto t^\alpha$ is subadditive, hence
\[
d_{\mathrm{reg}}(x,z)=\lambda\delta_{\mathrm{reg}}(x,z)^\alpha
\le \lambda(\delta_{\mathrm{reg}}(x,y)+\delta_{\mathrm{reg}}(y,z))^\alpha
\le d_{\mathrm{reg}}(x,y)+d_{\mathrm{reg}}(y,z).
\]
For a triangle involving the basepoint, note that $\delta_{\mathrm{reg}}(x,y)\le \rho(x)+\rho(y)$, hence
\[
d_{\mathrm{reg}}(x,y)=\lambda\delta_{\mathrm{reg}}(x,y)^\alpha
\le \lambda(\rho(x)+\rho(y))^\alpha
\le \lambda\rho(x)^\alpha+\lambda\rho(y)^\alpha
=d_{\mathrm{reg}}(x,\ast)+d_{\mathrm{reg}}(\ast,y).
\]
Separation follows because $\delta_{\mathrm{reg}}$ is a metric and $\rho(x)=0$ iff $x=0\in\mathcal C$.
\end{proof}

We henceforth regard $(\mathcal Y,d_{\mathrm{reg}})$ as a pointed metric space with basepoint $\ast$.  We now introduce the $\ell^p$ wedge construction for pointed metric spaces, a natural $\ell^p$-analogue of the usual gluing metric on a wedge sum; compare \cite[Definition~5.23 and Lemma~5.24]{BH} for the standard gluing metric and \cite[\S5.1 and Proposition~5.5]{CDT} for the corresponding $\ell^p$ construction.

\begin{definition}[\texorpdfstring{$\ell^p$}{lp} metric wedge]\label{def:lp-wedge}
Let $(M_1,d_1,m_0)$ and $(M_2,d_2,n_0)$ be pointed metric spaces, and fix $p\in[1,\infty]$.
The $\ell^p$ metric wedge $M_1\vee_p M_2$ is the quotient set of $M_1\sqcup M_2$ obtained by identifying
$m_0\sim n_0$, equipped with the function $d_{\vee,p}$ defined by
\[
d_{\vee,p}(x,y):=
\begin{cases}
d_1(x,y), & x,y\in M_1,\\[1mm]
d_2(x,y), & x,y\in M_2,\\[1mm]
\norm{\bigl(d_1(x,m_0),\,d_2(y,n_0)\bigr)}_{\ell^p}, & x\in M_1,\ y\in M_2,\\[1mm]
\norm{\bigl(d_1(y,m_0),\,d_2(x,n_0)\bigr)}_{\ell^p}, & x\in M_2,\ y\in M_1.
\end{cases}
\]
\end{definition}
The $\ell^p$ wedge construction defines a metric on the glued space.
\begin{lemma}\label{lem:lp-wedge-metric}
For every $p\in[1,\infty]$, the function $d_{\vee,p}$ is a metric on $M_1\vee_p M_2$.
\end{lemma}

\begin{proof}
Symmetry and nonnegativity are clear. Separation follows because:
(i) if $x,y$ lie in the same component, separation holds since $d_1$ and $d_2$ are metrics;
(ii) if $x\in M_1$ and $y\in M_2$, then
  $d_{\vee,p}(x,y)=\norm{(d_1(x,m_0),d_2(y,n_0))}_{\ell^p}=0$
  forces $d_1(x,m_0)=0$ and $d_2(y,n_0)=0$, hence $x=m_0$ and $y=n_0$, i.e.\ they represent the same glued point.

For the triangle inequality, fix $x,y,z\in M_1\vee_p M_2$ and consider cases.

\smallskip\noindent
Case 1: $x,y,z$ all lie in $M_1$ (or all in $M_2$). Then the triangle inequality holds by that of $d_1$ (or $d_2$).

\smallskip\noindent
Case 2: $x,y\in M_1$ and $z\in M_2$.
Then $
d_{\vee,p}(x,z)=\norm{(d_1(x,m_0),\,d_2(z,n_0))}_{\ell^p}.$
By the triangle inequality in $M_1$, we have
\[
(d_1(x,m_0),\,d_2(z,n_0))
\le (d_1(x,y)+d_1(y,m_0),\,d_2(z,n_0))
=(d_1(x,y),0)+(d_1(y,m_0),d_2(z,n_0))
\]
componentwise, and therefore (using Minkowski on $\R^2$)
\[
d_{\vee,p}(x,z)
\le \norm{(d_1(x,y),0)}_{\ell^p}+\norm{(d_1(y,m_0),d_2(z,n_0))}_{\ell^p}
= d_1(x,y)+d_{\vee,p}(y,z)
= d_{\vee,p}(x,y)+d_{\vee,p}(y,z).
\]

\smallskip\noindent
Case 3: $x\in M_1$ and $y,z\in M_2$. This is symmetric to Case 2.

\smallskip\noindent
Case 4: $x\in M_1$, $y\in M_2$, $z\in M_1$.
Then
\[
d_{\vee,p}(x,z)=d_1(x,z)\le d_1(x,m_0)+d_1(m_0,z)
=\norm{(d_1(x,m_0),0)}_{\ell^p}+\norm{(d_1(z,m_0),0)}_{\ell^p}.
\]
Also, by definition, $d_{\vee,p}(x,y)=\norm{(d_1(x,m_0),d_2(y,n_0))}_{\ell^p}\ge \norm{(d_1(x,m_0),0)}_{\ell^p}=d_1(x,m_0),$
and similarly $d_{\vee,p}(y,z)\ge d_1(z,m_0)$. Therefore
\[
d_{\vee,p}(x,z)\le d_1(x,m_0)+d_1(z,m_0)\le d_{\vee,p}(x,y)+d_{\vee,p}(y,z).
\]

All remaining configurations reduce to these by symmetry. Thus $d_{\vee,p}$ is a metric.
\end{proof}
The following theorem shows that BK metric spaces admit a canonical $\ell^p$ wedge decomposition.

\begin{theorem}
\label{thm:BK-lp-wedge}
Fix $\theta\in\mathcal C$, $\lambda>0$, $\alpha\in(0,1]$, and $p\in[1,\infty]$. Then the BK metric space $(\mathcal X, \beta^{BK}_{\theta,\lambda,p,\alpha})$ is isometric to the $\ell^p$ wedge  $\mathcal C\vee_p\mathcal Y$ of the pointed metric spaces $(\mathcal C,\beta,\theta)$
and $(\mathcal Y,d_{\mathrm{reg}},\ast)$,
via the identification $J:\mathcal X\to \mathcal C\vee_p\mathcal Y$ that is the identity on $\mathcal C\subseteq\mathcal X$ and on $\mathcal X\setminus\mathcal C\subseteq\mathcal Y$. Equivalently, for all $\phi,\psi\in\mathcal X$,

$$
\beta^{BK}_{\theta,\lambda,p,\alpha}(\phi,\psi)=
\begin{cases}
\beta(\phi,\psi), & \phi,\psi\in\mathcal C,\\[1mm]
d_{\mathrm{reg}}(\phi,\psi)=\lambda\bigl(\delta_{\mathrm{reg}}(\phi,\psi)\bigr)^\alpha, & \phi,\psi\notin\mathcal C,\\[1mm]
\norm{\bigl(\beta(\phi,\theta),\ d_{\mathrm{reg}}(\psi,\ast)\bigr)}_{\ell^p}
=\norm{\bigl(\beta(\phi,\theta),\ \lambda\rho(\psi)^\alpha\bigr)}_{\ell^p},
& \phi\in\mathcal C,\ \psi\notin\mathcal C,\\[1mm]
\norm{\bigl(\beta(\psi,\theta),\ d_{\mathrm{reg}}(\phi,\ast)\bigr)}_{\ell^p}
=\norm{\bigl(\beta(\psi,\theta),\ \lambda\rho(\phi)^\alpha\bigr)}_{\ell^p},
& \psi\in\mathcal C,\ \phi\notin\mathcal C.
\end{cases}
$$
\end{theorem}

\begin{proof}
We verify that the distance formulas on $\mathcal X$ match those on the wedge.

\smallskip\noindent
If $\phi,\psi\in\mathcal C$, then by construction $\widetilde\delta_{\mathrm{reg}}(\phi,\psi)=0$, and by Lemma~\ref{lem:kuratowski-isometry-anchor} $\norm{\kappa_\theta(\phi)-\kappa_\theta(\psi)}_\infty=\beta(\phi,\psi).$
Hence
\[
\beta^{BK}_{\theta,\lambda,p,\alpha}(\phi,\psi)
=\norm{(\beta(\phi,\psi),0)}_{\ell^p}
=\beta(\phi,\psi),
\]
which agrees with the wedge metric on the $\mathcal C$ component.

\smallskip\noindent
If $\phi,\psi\notin\mathcal C$, then $\kappa_\theta(\phi)=\kappa_\theta(\psi)=0$ and
$\widetilde\delta_{\mathrm{reg}}(\phi,\psi)=\delta_{\mathrm{reg}}(\phi,\psi)$, so
\[
\beta^{BK}_{\theta,\lambda,p,\alpha}(\phi,\psi)
=\norm{(0,\lambda\delta_{\mathrm{reg}}(\phi,\psi)^\alpha)}_{\ell^p}
=\lambda\delta_{\mathrm{reg}}(\phi,\psi)^\alpha
=d_{\mathrm{reg}}(\phi,\psi),
\]
which matches the wedge metric on the $\mathcal Y$ component.

\smallskip\noindent
In the cross case, assume $\phi\in\mathcal C$ and $\psi\notin\mathcal C$.
Then $\kappa_\theta(\psi)=0$, so $\norm{\kappa_\theta(\phi)-\kappa_\theta(\psi)}_\infty=\norm{\kappa_\theta(\phi)}_\infty=\beta(\phi,\theta),$
and $\widetilde\delta_{\mathrm{reg}}(\phi,\psi)=\rho(\psi)$ by definition of the collapse.
Thus
\[
\beta^{BK}_{\theta,\lambda,p,\alpha}(\phi,\psi)
=\norm{\bigl(\beta(\phi,\theta),\ \lambda\rho(\psi)^\alpha\bigr)}_{\ell^p}.
\]
On the wedge side, $J(\phi)\in\mathcal C$ and $J(\psi)\in\mathcal Y$, so by Definition~\ref{def:lp-wedge},
\[
d_{\vee,p}\bigl(J(\phi),J(\psi)\bigr)
=\norm{\bigl(\beta(\phi,\theta),\ d_{\mathrm{reg}}(\psi,\ast)\bigr)}_{\ell^p}
=\norm{\bigl(\beta(\phi,\theta),\ \lambda\rho(\psi)^\alpha\bigr)}_{\ell^p},
\]
which is exactly the same value. The remaining cross case is symmetric.
\end{proof}
\begin{proposition}
\label{prop:bk-attachment-closure}
Fix $\theta\in\mathcal C$.
Then:
\begin{enumerate}[label=\textup{(\roman*)}]
\item For every $\phi\in\mathcal C$, $\operatorname{dist}_{\beta^{BK}_{\theta,\lambda,p,\alpha}}(\phi,\mathcal X\setminus\mathcal C)=\beta(\phi,\theta)$.
\item $\overline{\mathcal X\setminus\mathcal C}^{\,\beta^{BK}_{\theta,\lambda,p,\alpha}}\cap \mathcal C=\{\theta\}$.
\item $\mathcal C$ is $\beta^{BK}_{\theta,\lambda,p,\alpha}$-closed in $\mathcal X$, and $\overline{\mathcal X\setminus\mathcal C}^{\,\beta^{BK}_{\theta,\lambda,p,\alpha}}=(\mathcal X\setminus\mathcal C)\ \cup\ \{\theta\}$.
\end{enumerate}
\end{proposition}

\begin{proof}
(i) If $\psi\notin\mathcal C$, then by Theorem~\ref{thm:BK-lp-wedge}, $\beta^{BK}_{\theta,\lambda,p,\alpha}(\phi,\psi)=\norm{(\beta(\phi,\theta),\lambda\rho(\psi)^\alpha)}_{\ell^p}\ge \beta(\phi,\theta).$
Thus $\operatorname{dist}(\phi,\mathcal X\setminus\mathcal C)\ge \beta(\phi,\theta)$.

Conversely, pick any nonzero state $\omega$ on $A$ and define $\psi_n(a):=\frac{i}{n}\,\omega(a)\,I_H$, as in the proof of Proposition~\ref{prop:anchors-nonequiv}.
Then $\psi_n\notin\mathcal C$ and 
$\rho(\psi_n)\le \|\psi_n\|_{cb}^{1/2}\to 0$ (Lemma~\ref{lem:deltareg-basic}).
Hence
\[
\beta^{BK}_{\theta,\lambda,p,\alpha}(\phi,\psi_n)
=\norm{(\beta(\phi,\theta),\lambda\rho(\psi_n)^\alpha)}_{\ell^p}\longrightarrow \beta(\phi,\theta),
\]
so $\operatorname{dist}(\phi,\mathcal X\setminus\mathcal C)\le \beta(\phi,\theta)$ and equality holds.

(ii) A point $\phi\in\mathcal C$ lies in $\overline{\mathcal X\setminus\mathcal C}^{\,\beta^{BK}_{\theta,\lambda,p,\alpha}}$ iff
$\operatorname{dist}_{\beta^{BK}_{\theta,\lambda,p,\alpha}}(\phi,\mathcal X\setminus\mathcal C)=0$.
By (i) this is equivalent to $\beta(\phi,\theta)=0$, i.e.\ $\phi=\theta$.

(iii) Let $\{\phi_n\}_{n}\subseteq \mathcal C$ and suppose that $\phi_n\to\psi$ in $\beta^{BK}_{\theta,\lambda,p,\alpha}.$
We claim that \(\psi\in\mathcal C\). Indeed, if $\psi\notin\mathcal C$, then by Theorem~\ref{thm:BK-lp-wedge}, for every \(n\),
$\beta^{BK}_{\theta,\lambda,p,\alpha}(\phi_n,\psi)=
\Bigl\|\bigl(\beta(\phi_n,\theta),\,\lambda\,\rho(\psi)^\alpha\bigr)\Bigr\|_{\ell^p}.$
Since $\psi\notin\mathcal C$, we have $\psi\neq 0$, because $0\in\mathcal C$, and hence $\rho(\psi)=\delta_{\mathrm{reg}}(\psi,0)>0.$
Therefore $\beta^{BK}_{\theta,\lambda,p,\alpha}(\phi_n,\psi)
\ge
\lambda\,\rho(\psi)^\alpha
>0$ for all $n$,
which contradicts \(\phi_n\to\psi\) in \(\beta^{BK}_{\theta,\lambda,p,\alpha}\).
Thus necessarily \(\psi\in\mathcal C\), proving that \(\mathcal C\) is closed.

Next we prove $\overline{\mathcal X\setminus\mathcal C}^{\,\beta^{BK}_{\theta,\lambda,p,\alpha}}
=
(\mathcal X\setminus\mathcal C)\cup\{\theta\}.$

The inclusion ``$\subseteq$'' follows immediately from part~(ii). 
For the reverse inclusion, it is enough to show that \(\theta\) belongs to the closure of
\(\mathcal X\setminus\mathcal C\). For this, let $\psi_n(a):=\frac{i}{n}\,\omega(a)\,I_H$ be as in the proof of part~(i). Apply the cross-case formula from Theorem~\ref{thm:BK-lp-wedge} with \(\phi=\theta\in\mathcal C\):
\[
\beta^{BK}_{\theta,\lambda,p,\alpha}(\theta,\psi_n)
=
\Bigl\|\bigl(\beta(\theta,\theta),\,\lambda\,\rho(\psi_n)^\alpha\bigr)\Bigr\|_{\ell^p}
=
\lambda\,\rho(\psi_n)^\alpha
\le
\lambda\,n^{-\alpha/2}.
\]
Hence $\beta^{BK}_{\theta,\lambda,p,\alpha}(\theta,\psi_n)\to 0,$ i.e., 
so $\psi_n\to\theta\) in \(\beta^{BK}_{\theta,\lambda,p,\alpha}$.
Therefore $\theta\in
\overline{\mathcal X\setminus\mathcal C}^{\,\beta^{BK}_{\theta,\lambda,p,\alpha}}.$
\end{proof}

\begin{example}
Assume $(A,H)=(\C,\C)$. Every $\phi\in\CB(\C,\C)$ is multiplication by a scalar: $\phi=\phi_z$, $\phi_z(\lambda)=z\lambda$, $z\in\C$.
Under the identification $\mathcal X\cong\C$, the cb-norm is $\|\phi_z\|_{cb}=|z|$ and
$\mathcal C=\CP(\C,\C)\cong [0,\infty)$.

Fix a cp map $\theta=\phi_t$ with $t\ge 0$. For the max-glue BK metric $d_\theta=\beta^{BK}_{\theta,\lambda,\infty,\alpha}$,
\eqref{eq:1-dim} and Theorem~\ref{thm:BK-lp-wedge} give:
\begin{enumerate}[label=\textup{(\roman*)}]
\item If $c_1,c_2\ge 0$, then $d_\theta(\phi_{c_1},\phi_{c_2})=|\sqrt{c_1}-\sqrt{c_2}|$.
\item If $z,w\in\C\setminus[0,\infty)$, then $d_\theta(\phi_z,\phi_w)=\lambda\bigl(\delta_{\mathrm{reg}}(\phi_z,\phi_w)\bigr)^\alpha$.
\item If $c\ge 0$ and $z\in\C\setminus[0,\infty)$, then
\[
d_\theta(\phi_c,\phi_z)=\max\Bigl\{|\sqrt c-\sqrt t|,\ \lambda\bigl(\rho(\phi_z)\bigr)^\alpha\Bigr\},
\qquad
\rho(\phi_z)=\delta_{\mathrm{reg}}(\phi_z,0).
\]
\end{enumerate}
\end{example}

\section{Vietoris--Rips and \v{C}ech complexes for BK metric spaces}

Let $A$ be a unital $C^*$-algebra and $H$ a Hilbert space. Fix $\lambda>0$, $\alpha\in(0,1]$, and $p\in[1,\infty]$, and set $d_\theta:=\beta^{BK}_{\theta,\lambda,p,\alpha}$, for simplicity. By the $\ell^p$ wedge decomposition (Theorem~\ref{thm:BK-lp-wedge}), we identify
$(\mathcal X,d_\theta)\ \cong\ (\mathcal C,\beta,\theta)\ \vee_p\ (\mathcal Y,d_{\mathrm{reg}},\ast),$
so that: (i) on $\mathcal C$ we have $d_\theta=\beta$; (ii)
on $\mathcal Y^\circ:=\mathcal Y\setminus\{\ast\}\cong \mathcal X\setminus\mathcal C$ we have
      $d_\theta=d_{\mathrm{reg}}=\lambda\,\delta_{\mathrm{reg}}^\alpha$; and 
(iii) for $x\in\mathcal C$ and $y\in\mathcal Y$,
\begin{equation}\label{eq:cross-distance}
d_\theta(x,y)=\bigl\|\bigl(\beta(x,\theta),\ d_{\mathrm{reg}}(y,\ast)\bigr)\bigr\|_{\ell^p}.
\end{equation}

It is convenient to introduce the radial functions $r_{\mathcal C}(x):=\beta(x,\theta)$ for all $x\in\mathcal C$, and $r_{\mathcal Y}(y):=d_{\mathrm{reg}}(y,\ast)$ for all $y\in\mathcal Y.$ 
With this notation, \eqref{eq:cross-distance} becomes
$d_\theta(x,y)=\bigl\|\bigl(r_{\mathcal C}(x), r_{\mathcal Y}(y)\bigr)\bigr\|_{\ell^p},$ for all $ x\in\mathcal C,\ y\in\mathcal Y.$

\subsection{Vietoris--Rips complexes}
We now study the Vietoris--Rips complexes associated to the BK metric.
\begin{definition}[cf. {\cite{Hausmann}\cite{Gromov}}]
Let $(Z,d)$ be a metric space and $t\ge 0$.
The Vietoris--Rips complex $\mathrm{VR}_t(Z,d)$ is the abstract simplicial complex with vertex set $Z$
such that a finite subset $\sigma\subseteq Z$ spans a simplex iff $d(z,z')\le t$ for all $z,z'\in\sigma.$
Equivalently, $\sigma$ is a simplex iff $\diam(\sigma)\le t$.
\end{definition}
We begin by characterizing cross edges in terms of the radial coordinates.
\begin{lemma}
\label{lem:cross-edge}
Let $x\in\mathcal C$ and $y\in \mathcal Y^\circ$. Then $\{x,y\}$ is an edge of $\mathrm{VR}_t(\mathcal X,d_\theta)$ iff $\bigl\|\bigl(r_{\mathcal C}(x), r_{\mathcal Y}(y)\bigr)\bigr\|_{\ell^p}
\le\ t.$
\end{lemma}

\begin{proof}This follows directly from the cross-distance formula \eqref{eq:cross-distance}.
\end{proof}
We next characterize mixed simplices in terms of a maximal radial constraint.
\begin{proposition}
\label{prop:rips-mixed}
Fix $t\ge 0$ and let $\sigma\subseteq\mathcal C$ and $\tau\subseteq\mathcal Y^\circ$ be finite, nonempty subsets.
Set
\[
A:=\max_{x\in\sigma} r_{\mathcal C}(x),
\qquad
B:=\max_{y\in\tau} r_{\mathcal Y}(y).
\]
Then $\sigma\cup\tau$ is a simplex of $\mathrm{VR}_t(\mathcal X,d_\theta)$ if and only if the following conditions are satisfied:
\begin{enumerate}
\item $\sigma$ is a simplex of $\mathrm{VR}_t(\mathcal C,\beta)$;
\item $\tau$ is a simplex of $\mathrm{VR}_t(\mathcal Y^\circ,d_{\mathrm{reg}})$;
\item $\bigl\|\bigl(A, B\bigr)\bigr\|_{\ell^p}\le t$.
\end{enumerate}
\end{proposition}

\begin{proof}
If $\sigma\cup\tau$ is a simplex in $\mathrm{VR}_t(\mathcal X,d_\theta)$, then in particular all pairwise distances
between vertices in $\sigma$ are $\le t$, which is exactly (1) since $d_\theta|_{\mathcal C}=\beta$.
Similarly all pairwise distances between vertices in $\tau$ are $\le t$, giving (2) since $d_\theta|_{\mathcal Y}=d_{\mathrm{reg}}$.

For the cross pairs, $\sigma\cup\tau$ being a simplex means that for all $x\in\sigma$ and $y\in\tau$, $d_\theta(x,y)=\bigl\|\bigl(r_{\mathcal C}(x), r_{\mathcal Y}(y)\bigr)\bigr\|_{\ell^p}\le t.$

Since$ \bigl\|\bigl(\cdot, \cdot\bigr)\bigr\|_{\ell^p}$ is increasing in each variable, this is equivalent to the single inequality
\[ \Bigl\|\Bigl(\max_{x\in\sigma}r_{\mathcal C}(x),\ \max_{y\in\tau} r_{\mathcal Y}(y)\Bigr)\Bigr\|_{\ell^p}
\le t,
\]
i.e.\  $\bigl\|\bigl(A, B\bigr)\bigr\|_{\ell^p}\le t$. (Necessity uses the particular pair $(x_0,y_0)$ where the maxima are attained.)

Conversely, if (1) and (2) hold and $\bigl\|\bigl(A, B\bigr)\bigr\|_{\ell^p}\le t$, then for each cross pair $(x,y)$ we have
$r_{\mathcal C}(x)\le A$ and $r_{\mathcal Y}(y)\le B$, hence
$d_\theta(x,y)= \bigl\|\bigl(r_{\mathcal C}(x),r_{\mathcal Y}(y)\bigr)\bigr\|_{\ell^p}
\le \bigl\|\bigl(A, B\bigr)\bigr\|_{\ell^p}
\le t.$
Thus all pairwise distances in $\sigma\cup\tau$ are $\le t$, so it is a simplex.
\end{proof}

To describe the mixed part more structurally, we introduce radial sublevel sets and a join-type construction.

\begin{definition}
\label{def:rips-join}
(i) For $u\ge 0$ define the radial sublevel sets
\[
\mathcal C^{\le u}:=\{x\in\mathcal C:\ r_{\mathcal C}(x)\le u\},
\qquad
\mathcal Y^{\le u}:=\{y\in\mathcal Y^\circ:\ r_{\mathcal Y}(y)\le u\}.
\]
(ii) Given a simplicial complex \(K\) on vertex set \(V(K)\) and a subset \(W\subseteq V(K)\), we write
\(K|_W\) for the induced subcomplex on \(W\).
In particular, we set
\[
\mathrm{VR}_t(\mathcal C^{\le u})
:=\mathrm{VR}_t(\mathcal C,\beta)\big|_{\mathcal C^{\le u}},
\qquad
\mathrm{VR}_t(\mathcal Y^{\le v})
:=\mathrm{VR}_t(\mathcal Y^\circ,d_{\mathrm{reg}})\big|_{\mathcal Y^{\le v}}.
\]
(iii) Let \(K,L\) be simplicial complexes on disjoint vertex sets.
Define their mixed join to be the collection
\[
K\ \star\ L
:=\bigl\{\sigma\cup\tau:\ \sigma\in K,\ \tau\in L,\ \sigma\neq\varnothing,\ \tau\neq\varnothing\bigr\}.
\]
(Equivalently, \(K\star L\) is the subcollection of simplices of the usual join \(K*L\) that meet both factors.)
\end{definition}
The following theorem identifies the mixed simplices as a union of constrained joins.
\begin{theorem}
\label{thm:rips-join}
For each \(t\ge 0\), the set of simplices of \(\mathrm{VR}_t(\mathcal X,d_\theta)\) that meet both
\(\mathcal C\) and \(\mathcal Y^\circ\) is exactly
\[
\bigcup_{\substack{u,v\ge 0\\ \norm{(u,v)}_{\ell^p}\le t}}
\Bigl(\mathrm{VR}_t(\mathcal C^{\le u})\ \star\ \mathrm{VR}_t(\mathcal Y^{\le v})\Bigr).
\]
In particular, when \(p=\infty\) one has the global decomposition
\begin{equation}\label{eq:rips-pinf-decomp}
\mathrm{VR}_t(\mathcal X,d_\theta)
=
\mathrm{VR}_t(\mathcal C,\beta)\ \cup\ \mathrm{VR}_t(\mathcal Y^\circ,d_{\mathrm{reg}})
\ \cup\ \Bigl(\mathrm{VR}_t(\mathcal C^{\le t}) \star \mathrm{VR}_t(\mathcal Y^{\le t})\Bigr).
\end{equation}
\end{theorem}

\begin{proof}
Fix \(t\ge 0\), and write \(\mathrm{VR}_t:=\mathrm{VR}_t(\mathcal X,d_\theta)\), for simplicity.

\smallskip
\noindent

($\subseteq$) Let \(\eta\in \mathrm{VR}_t\) be a simplex that meets both \(\mathcal C\) and \(\mathcal Y^\circ\).
Set \(\sigma:=\eta\cap\mathcal C\) and \(\tau:=\eta\cap\mathcal Y^\circ\), so \(\sigma,\tau\) are finite and nonempty and
\(\eta=\sigma\cup\tau\) (disjoint union of vertex sets).
Define
\[
u:=\max_{x\in\sigma} r_{\mathcal C}(x),\qquad v:=\max_{y\in\tau} r_{\mathcal Y}(y).
\]
Since \(\eta\) is a simplex, all pairwise distances in \(\eta\) are \(\le t\).
In particular, \(d_\theta|_{\mathcal C}=\beta\) implies \(\sigma\in \mathrm{VR}_t(\mathcal C,\beta)\),
and \(d_\theta|_{\mathcal Y^\circ}=d_{\mathrm{reg}}\) implies \(\tau\in \mathrm{VR}_t(\mathcal Y^\circ,d_{\mathrm{reg}})\).
Moreover, for any cross pair \((x,y)\in\sigma\times\tau\),
\[
d_\theta(x,y)=\norm{\bigl(r_{\mathcal C}(x),r_{\mathcal Y}(y)\bigr)}_{\ell^p}\le t.
\]
By monotonicity of \(\ell^p\)-norm in each coordinate, this yields \(\norm{(u,v)}_{\ell^p}\le t\).
Also, by definition of \(u\) and \(v\), we have \(\sigma\subseteq\mathcal C^{\le u}\) and \(\tau\subseteq\mathcal Y^{\le v}\),
hence \(\sigma\in \mathrm{VR}_t(\mathcal C^{\le u})\) and \(\tau\in \mathrm{VR}_t(\mathcal Y^{\le v})\).
Therefore \(\eta=\sigma\cup\tau\in \mathrm{VR}_t(\mathcal C^{\le u})\star \mathrm{VR}_t(\mathcal Y^{\le v})\)
for some \((u,v)\) with \(\norm{(u,v)}_{\ell^p}\le t\), proving the containment.

\smallskip
\noindent

($\supseteq$) Fix \(u,v\ge 0\) with \(\norm{(u,v)}_{\ell^p}\le t\) and let
\(\eta=\sigma\cup\tau\in \mathrm{VR}_t(\mathcal C^{\le u})\star \mathrm{VR}_t(\mathcal Y^{\le v})\).
Then \(\sigma,\tau\) are nonempty, \(\sigma\subseteq \mathcal C^{\le u}\), \(\tau\subseteq \mathcal Y^{\le v}\),
\(\sigma\in\mathrm{VR}_t(\mathcal C,\beta)\), and \(\tau\in\mathrm{VR}_t(\mathcal Y^\circ,d_{\mathrm{reg}})\).
Hence all within-\(\mathcal C\) and within-\(\mathcal Y^\circ\) distances inside \(\eta\) are \(\le t\).

For a cross pair \(x\in\sigma\), \(y\in\tau\), we have \(r_{\mathcal C}(x)\le u\) and \(r_{\mathcal Y}(y)\le v\), so
\[
d_\theta(x,y)=\norm{\bigl(r_{\mathcal C}(x),r_{\mathcal Y}(y)\bigr)}_{\ell^p}
\le \norm{(u,v)}_{\ell^p}\le t.
\]
Thus all pairwise distances in $\eta$ are $\le t$, i.e. \(\eta\in \mathrm{VR}_t(\mathcal X,d_\theta)\),
and by construction $\eta$ meets both sides. This proves equality of mixed parts.

\smallskip
\noindent

Finally, assume $p=\infty$. Then $\norm{(u,v)}_{\ell^\infty}\le t$ iff $u\le t$ and $v\le t$.
Hence the mixed part of \(\mathrm{VR}_t(\mathcal X,d_\theta)\) is exactly
$\mathrm{VR}_t(\mathcal C^{\le t})\star \mathrm{VR}_t(\mathcal Y^{\le t})$.
Since the full complex $\mathrm{VR}_t(\mathcal X,d_\theta)$ consists of:
(i) simplices entirely in $\mathcal C$, i.e. $\mathrm{VR}_t(\mathcal C,\beta)$;
(ii) simplices entirely in $\mathcal Y^\circ$, i.e. $\mathrm{VR}_t(\mathcal Y^\circ,d_{\mathrm{reg}})$;
and (iii) mixed simplices,
we obtain \eqref{eq:rips-pinf-decomp}. 
\end{proof}

\begin{remark}
The cross edges (and hence all mixed simplices) are governed purely by the two scalar functions
$r_{\mathcal C}$ and $r_{\mathcal Y}$.
Thus, in the Rips filtration $\{\mathrm{VR}_t(\mathcal X,d_\theta)\}_{t\ge 0}$, the two ``large'' pieces
$\mathrm{VR}_t(\mathcal C)$ and $\mathrm{VR}_t(\mathcal Y^\circ)$ are always present, and the only way they interact is through
a radius-thresholded, join-like family of simplices near the glued basepoints.
This is a strong structural simplification compared to an arbitrary metric space.
\end{remark}
We restate \eqref{eq:rips-pinf-decomp} for finite subsets $S\subset\mathcal X$. For this, set $S_{\mathcal C}:=S\cap\mathcal C,$ $S_{\mathcal Y}:=S\cap(\mathcal X\setminus\mathcal C),$ and consider the sublevel vertex sets
\[
S_{\mathcal C}^{\le t}:=\{x\in S_{\mathcal C}: r_{\mathcal C}(x)\le t\},
\qquad
S_{\mathcal Y}^{\le t}:=\{y\in S_{\mathcal Y}: r_{\mathcal Y}(y)\le t\}.
\]
For $p=\infty$, the Rips complex decomposes exactly as a simplicial join on finite clouds.
\begin{corollary}
\label{prop:rips_exact_pinf}
For every finite $S\subset\mathcal X$ and every $t\ge 0$,
\[
\VR_t(S,d_\theta)
=
\VR_t(S_{\mathcal C},\beta)\ \cup\ \VR_t(S_{\mathcal Y},d_{\mathrm{reg}})\
\cup\ \Bigl(\VR_t(S_{\mathcal C}^{\le t},\beta)\ \star\ \VR_t(S_{\mathcal Y}^{\le t},d_{\mathrm{reg}})\Bigr).
\]
\end{corollary}

\begin{remark}
Corollary~\ref{prop:rips_exact_pinf} makes $\VR_t(S,d_\theta)$ explicitly computable once one knows:
 the distance matrix of $\beta$ on $S_{\mathcal C}$ and the radius values $r_{\mathcal C}$, and 
the distance matrix of $d_{\mathrm{reg}}$ on $S_{\mathcal Y}$ and the radius values $r_{\mathcal Y}$.
No further mixed distance computations are needed.
\end{remark}

BK Rips complexes can have nontrivial homology already at the level of
finite point clouds, and that the effect is driven purely by the wedge join term, as shown below.

\begin{corollary}[A universal loop in $\VR_t$]\label{ex:universal_loop}
Fix $p=\infty$ and a scale $t>0$.
Assume a finite point cloud $S\subset\mathcal X$ has the following properties:
\begin{enumerate}
\item $S_{\mathcal C}^{\le t}=\{x_1,x_2\}$ and $\beta(x_1,x_2)>t$ (so $\VR_t(S_{\mathcal C}^{\le t},\beta)$ is two isolated vertices);
\item $S_{\mathcal Y}^{\le t}=\{y_1,y_2\}$ and $d_{\mathrm{reg}}(y_1,y_2)>t$ (so $\VR_t(S_{\mathcal Y}^{\le t},d_{\mathrm{reg}})$ is two isolated vertices);
\item there are no other vertices in $S$ (i.e.\ $S=S_{\mathcal C}^{\le t}\cup S_{\mathcal Y}^{\le t}$).
\end{enumerate}
Then $\VR_t(S,d_\theta)$ is the clique complex of the complete bipartite graph $K_{2,2}$,
hence is a $4$-cycle and has homotopy type $S^1$.
\end{corollary}

\begin{proof}
By Corollary~\ref{prop:rips_exact_pinf}, since there are no within-side edges at scale $t$,
the only simplices are those coming from the join term:
every mixed pair $\{x_i,y_j\}$ is an edge because both radii are $\le t$,
and there are no triangles because there are no within-side edges.
Thus the 1-skeleton is exactly the graph $K_{2,2}$, and the flag (clique) complex is the graph itself.
A connected 4-cycle has homotopy type $S^1$.
\end{proof}

\usetikzlibrary{positioning}

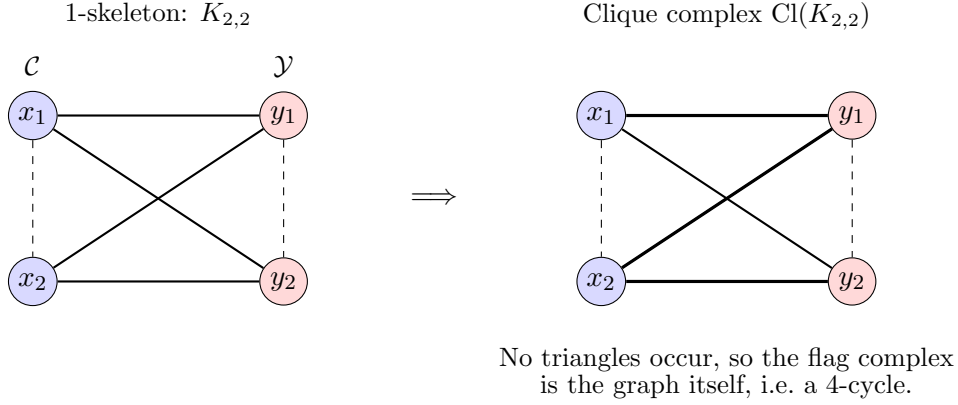
\begin{figure}[t]
\centering
\begin{tikzpicture}[scale=1.1,
    cp/.style={circle,draw,fill=blue!15,inner sep=2pt},
    ncp/.style={circle,draw,fill=red!15,inner sep=2pt}
]

\node[draw=none] at (1.5,2.2) {\small 1-skeleton: $K_{2,2}$};

\node[cp]  (x1L) at (0,1) {$x_1$};
\node[cp]  (x2L) at (0,-1) {$x_2$};
\node[ncp] (y1L) at (3,1) {$y_1$};
\node[ncp] (y2L) at (3,-1) {$y_2$};

\draw[thick] (x1L) -- (y1L);
\draw[thick] (x1L) -- (y2L);
\draw[thick] (x2L) -- (y1L);
\draw[thick] (x2L) -- (y2L);

\draw[dashed] (x1L) -- (x2L);
\draw[dashed] (y1L) -- (y2L);

\node[draw=none] at (0,1.6) {\small $\mathcal C$};
\node[draw=none] at (3,1.6) {\small $\mathcal Y$};

\node[draw=none] at (4.8,0) {$\Longrightarrow$};

\node[draw=none] at (8.3,2.2) {\small Clique complex $\mathrm{Cl}(K_{2,2})$};

\node[cp]  (x1R) at (6.8,1) {$x_1$};
\node[cp]  (x2R) at (6.8,-1) {$x_2$};
\node[ncp] (y1R) at (9.8,1) {$y_1$};
\node[ncp] (y2R) at (9.8,-1) {$y_2$};

\draw[very thick] (x1R) -- (y1R) -- (x2R) -- (y2R) -- cycle;

\draw[thick] (x1R) -- (y2R);
\draw[thick] (x2R) -- (y1R);

\draw[dashed] (x1R) -- (x2R);
\draw[dashed] (y1R) -- (y2R);

\node[draw=none,align=center] at (8.3,-2.1) {\small No triangles occur, so the flag complex\\[-1mm]
\small is the graph itself, i.e.\ a $4$-cycle.};

\end{tikzpicture}

\caption{Corollary~\ref{ex:universal_loop}: the BK complex at scale $t$ has 1-skeleton equal to the complete bipartite graph $K_{2,2}$, with bipartition $\{x_1,x_2\}\subset \mathcal C$ and $\{y_1,y_2\}\subset \mathcal Y$. Since there are no within-side edges, there are no triangles, and hence the flag (clique) complex adds no $2$-simplices. Therefore $\VR_t(S,d_\theta)$ is the graph itself, namely a $4$-cycle, and has homotopy type $S^1$.}
\label{fig:bk-k22-flag}
\end{figure}
The same mechanism produces higher-rank first homology.

\begin{corollary}
\label{ex:Km n}
Fix $p=\infty$ and $t>0$. Suppose
\[
S_{\mathcal C}^{\le t}=\{x_1,\dots,x_m\},\quad S_{\mathcal Y}^{\le t}=\{y_1,\dots,y_n\},
\]
and assume \emph{no} within-side edges at scale $t$:
$\beta(x_i,x_{i'})>t$ for $i\neq i'$ and $d_{\mathrm{reg}}(y_j,y_{j'})>t$ for $j\neq j'$.
Assume again that $S=S_{\mathcal C}^{\le t}\cup S_{\mathcal Y}^{\le t}$.

Then $\VR_t(S,d_\theta)$ is the graph $K_{m,n}$ viewed as a 1-dimensional simplicial complex.
In particular, for $m,n\ge 1$ it is connected and
\[
\mathrm{rank}\,H_1\bigl(\VR_t(S,d_\theta);\Z\bigr)
=
(m-1)(n-1),
\]
so $\VR_t(S,d_\theta)\simeq \bigvee^{(m-1)(n-1)} S^1$.
\end{corollary}

\begin{proof}
Exactly as in Corollary~\ref{ex:universal_loop}, all edges are mixed edges and there are no higher simplices.
Thus $\VR_t(S,d_\theta)$ is the graph $K_{m,n}$.
For a connected finite graph, $H_1$ has rank $E-V+1$.
Here $E=mn$ and $V=m+n$, so $\mathrm{rank}\,H_1=mn-(m+n)+1=(m-1)(n-1)$.
\end{proof}

\begin{remark}
Corollaries~\ref{ex:universal_loop}--\ref{ex:Km n} show a qualitative difference between BK Rips complexes and the
Rips complex of a single ``side'':
even if the CP-side Rips complex and the non-CP-side Rips complex are both totally disconnected at scale $t$,
the BK join term can create loops (and in general higher-dimensional homology, if within-side simplices occur).
\end{remark}
\subsubsection{A finite-dimensional CP-side computation: $A=\C^n$, $H=\C$}

We now compute the CP-side Bures metric explicitly in a standard finite-dimensional commutative case,
which then lets one compute the BK Rips complex for point clouds whose CP-vertices lie in this cone.

Let $A=\C^n$ with pointwise operations and involution, and let $H=\C$.
Every linear map $\varphi:A\to\C$ has the form $\varphi(a_1,\dots,a_n)=\sum_{k=1}^n z_k a_k$
for a unique $z=(z_1,\dots,z_n)\in\C^n.$
Such a $\varphi$ is positive (hence CP) iff $z_k\in[0,\infty)$ for all $k$.
Thus $\mathcal C=\CP(\C^n,\C)\ \cong\ [0,\infty)^n.$
Under this identification, the Bures distance on $\mathcal C$ is
\begin{equation}\label{eq:hellinger}
\beta(z,w)=\norm{\sqrt{z}-\sqrt{w}}_2,
\qquad z,w\in[0,\infty)^n,
\end{equation}
where $\sqrt{z}:=(\sqrt{z_1},\dots,\sqrt{z_n})$ and $\norm{\cdot}_2$ denotes the Euclidean norm on $\R^n$.

\begin{example}[Explicit CP-side Rips computation in dimension $2$]\label{ex:cp_rips_dim2}
Let $n=2$ and consider three CP points $z^{(1)}=(1,0),$ $z^{(2)}=(0,1),$ $z^{(3)}=(1,1)\in[0,\infty)^2.$
By \eqref{eq:hellinger},
\[
\beta(z^{(1)},z^{(2)})=\sqrt{2},\qquad
\beta(z^{(1)},z^{(3)})=1,\qquad
\beta(z^{(2)},z^{(3)})=1.
\]
Hence the CP-side Vietoris--Rips complexes $\VR_t(\{z^{(1)},z^{(2)},z^{(3)}\},\beta)$ are:
\begin{enumerate}
\item[(i)] for $0\le t<1$: three isolated vertices;
\item[(ii)] for $1\le t<\sqrt{2}$: a ``$V$''-shaped tree with edges $\{z^{(1)},z^{(3)}\}$ and $\{z^{(2)},z^{(3)}\}$;
\item[(iii)] for $t\ge\sqrt{2}$: the full $2$-simplex on three vertices.
\end{enumerate}
In particular, the CP-side Rips homology has no $H_1$ at any scale.
\end{example}

\begin{remark}[Adding non-CP vertices produces BK loops at intermediate scales]
Take the CP point cloud of Example~\ref{ex:cp_rips_dim2} and add two non-CP vertices $y_1,y_2\in\mathcal X\setminus\mathcal C$
with $d_{\mathrm{reg}}(y_1,y_2)>t$ but $r_{\mathcal Y}(y_1),r_{\mathcal Y}(y_2)\le t$.
Choose an anchor $\theta$ so that $r_{\mathcal C}(z^{(1)}),r_{\mathcal C}(z^{(2)})\le t$ but $r_{\mathcal C}(z^{(3)})>t$, and suppose $\beta(z^{(1)}, z^{(2)})>t$.
Then Corollary~\ref{prop:rips_exact_pinf} shows that at scale $t$ the join term contributes a copy of $K_{2,2}$,
hence creates an $H_1$-class (as in Corollary~\ref{ex:universal_loop}), even though the CP-side alone has no loops.
\end{remark}
\begin{remark}
For general $p<\infty$ the only difference from the max-glue case is the cross-distance rule:
\[
d_\theta(x,y)=\bigl(\beta(x,\theta)^p+d_{\mathrm{reg}}(y,\ast)^p\bigr)^{1/p},
\qquad x\in\mathcal C,\ y\in\mathcal Y,
\]
so a mixed edge at scale $t$ occurs iff $(r_{\mathcal C}(x))^p+(r_{\mathcal Y}(y))^p\le t^p$.
Accordingly, in the Rips mixed-simplex criterion, the join term is no longer based on simple sublevel sets
$r_{\mathcal C}\le t$ and $r_{\mathcal Y}\le t$, but on the constraint
\[
\max_{x\in\sigma} r_{\mathcal C}(x)^p\ +\ \max_{y\in\tau} r_{\mathcal Y}(y)^p\ \le\ t^p.
\]
Nevertheless, for a finite point cloud this is still fully computable from the two lists of radii.
\end{remark}

\subsubsection{A finite-dimensional computation on a mixed cloud: $A=M_n$, $H=\C^m$}

Let $\tau_n$ be the normalized trace state on $M_n$.
Define the (unital) completely depolarizing CP map
\begin{equation}\label{eq:depolarizing}
\Theta(a):=\tau_n(a)\,I_m,\qquad a\in M_n.
\end{equation}
Then $\Theta\in\CP(M_n,M_m)$ and $\norm{\Theta}_{\cb}=1$. For $c\ge 0$ define $\Theta_c:=c\,\Theta \in \CP(M_n,M_m)$.

\begin{lemma}
\label{lem:bures-ray}
$\beta(\Theta_c,\Theta_d)=\bigl|\sqrt{c}-\sqrt{d}\bigr|,$ for all $c,d\ge 0$,
\end{lemma}

\begin{proof}
Let $\Omega:=\frac1{\sqrt n}\sum_{j=1}^n e_j\otimes e_j\in \C^n\otimes\C^n$, so that
$\tau_n(a)=\langle \Omega,(a\otimes I_n)\Omega\rangle$.
Set $K:=\C^n\otimes\C^n\otimes \C^m$ and $\pi(a):=a\otimes I_n\otimes I_m$.
Define $V:\C^m\to K$ by $Vh:=\Omega\otimes h$. Then $V^*\pi(a)V=\tau_n(a)I_m=\Theta(a)$.
Thus $\sqrt c\,V$ implements $\Theta_c$ and $\sqrt d\,V$ implements $\Theta_d$ in the same representation $\pi$. Hence $
\beta(\Theta_c,\Theta_d)\le \norm{\sqrt c\,V-\sqrt d\,V}
=|\sqrt c-\sqrt d|\,\norm{V}
=|\sqrt c-\sqrt d|.$

For the reverse inequality, let $(\pi',W_c)$ implement $\Theta_c$.
Then $\norm{W_c}^2=\norm{\Theta_c(I_n)}=\norm{cI_m}=c$, so $\norm{W_c}=\sqrt c$; similarly any implementer of $\Theta_d$
has norm $\sqrt d$. Therefore for any common representation, $\norm{W_c-W_d}\ \ge\ \bigl|\norm{W_c}-\norm{W_d}\bigr|=|\sqrt c-\sqrt d|.$

Taking infima yields $\beta(\Theta_c,\Theta_d)\ge |\sqrt c-\sqrt d|$.
\end{proof}

\begin{remark}
Thus the ray $\{\Theta_c:c\ge 0\}\subset \CP(M_n,M_m)$ is isometric to $[0,\infty)$
via the coordinate $c\mapsto \sqrt c$. This gives a fully explicit CP-side geometry for any $(n,m)$.
\end{remark}
\begin{example}\label{ex:go1}
Fix the three CP maps $x_0:=\Theta_0=0,$ $x_1:=\Theta_1=\Theta$ and $x_4:=\Theta_4=4\Theta$,
and set $S_{\mathcal C}:=\{x_0,x_1,x_4\}\subset \CP(M_n,M_m)$. Then \begin{equation}\label{eq:bures-matrix}
\beta(x_0,x_1)=1,\qquad \beta(x_1,x_4)=1,\qquad \beta(x_0,x_4)=2.
\end{equation} by Lemma~\ref{lem:bures-ray}. As a result, we have:

\begin{enumerate}
\item[(i)] If $0\le t<1$, then  $\VR_t(S_{\mathcal C},\beta)$ is three isolated vertices.
\item[(ii)] If $1\le t<2$, then  $\VR_t(S_{\mathcal C},\beta)$ is the path graph $x_0$--$x_1$--$x_4$ (a tree).
\item[(iii)] If $t\ge 2$, then  $\VR_t(S_{\mathcal C},\beta)$ is the full $2$-simplex on $\{x_0,x_1,x_4\}$ (contractible).
\end{enumerate}

Indeed, edges appear exactly when the corresponding pairwise distance is $\le t$.
By \eqref{eq:bures-matrix}, the two short edges appear at $t\ge 1$ and the long edge appears at $t\ge 2$.
When all three edges are present, the Rips complex is the clique complex of $K_3$, i.e.\ a filled triangle.
\end{example}

\usetikzlibrary{positioning}

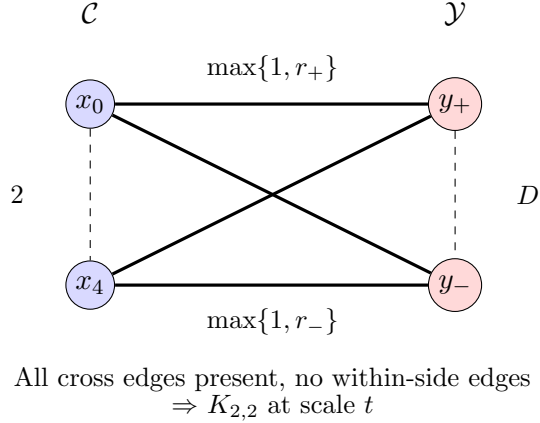
\begin{figure}[t]
\centering
\begin{tikzpicture}[scale=1.2,
    cp/.style={circle,draw,fill=blue!15,inner sep=2pt},
    ncp/.style={circle,draw,fill=red!15,inner sep=2pt}
]


\node[cp] (x0) at (0,1) {$x_0$};
\node[cp] (x4) at (0,-1) {$x_4$};

\node[ncp] (yp) at (4,1) {$y_+$};
\node[ncp] (ym) at (4,-1) {$y_-$};

\node[draw=none] at (0,2) {$\mathcal C$};
\node[draw=none] at (4,2) {$\mathcal Y$};


\draw[very thick] (x0) -- (yp);
\draw[very thick] (x0) -- (ym);
\draw[very thick] (x4) -- (yp);
\draw[very thick] (x4) -- (ym);

\draw[dashed] (x0) -- (x4);
\draw[dashed] (yp) -- (ym);


\node at (2,1.4) {\small $\max\{1,r_+\}$};
\node at (2,-1.4) {\small $\max\{1,r_-\}$};

\node at (-0.8,0) {\small $2$};
\node at (4.8,0) {\small $D$};


\node[draw=none,align=center] at (2,-2.2)
{\small All cross edges present, no within-side edges\\[-1mm]
 \small $\Rightarrow$ $K_{2,2}$ at scale $t$};

\end{tikzpicture}

\caption{The mixed point cloud $S=\{x_0,x_4,y_+,y_-\}$ from Example~\ref{ex:go2}. 
The CP vertices $x_0,x_4$ are at Bures distance $2$, while the non-CP vertices $y_+,y_-$ are at $d_{\mathrm{reg}}$-distance $D$. 
Each cross distance is $\max\{1,r_\pm\}$. 
Under the condition $\max\{1,r_+,r_-\}\le t<\min\{2,D\}$, all cross edges are present and no within-side edges occur. 
Thus the Vietoris--Rips complex $\VR_t(S,d_\theta)$ has 1-skeleton $K_{2,2}$ and is homotopy equivalent to $S^1$ (Theorem~\ref{thm:bk-rips-loop}).}
\label{fig:bk-mixed-cloud}

\end{figure}

We now illustrate the BK wedge effect on a finite cloud in $\CB(M_n,M_m)$.

\begin{example}\label{ex:go2}
Define the completely bounded maps $y_{+}:=i\Theta,$ $y_{-}:=-i\Theta.$ Then  $\norm{y_{\pm}}_{\cb}=1$, but these are not completely positive.
We define the (unknown but finite) parameters
\begin{equation}\label{eq:params}
r_+:=d_{\mathrm{reg}}(y_+,\ast),\quad r_{-}:=d_{\mathrm{reg}}(y_-,\ast),
\quad
D:=d_{\mathrm{reg}}(y_+,y_-)=\lambda\,\delta_{\mathrm{reg}}(y_+,y_-)^\alpha.
\end{equation}
They satisfy $0<r_{\pm}<\infty$, and $0\le D\le r_++r_{-}\le 2\max\{r_+,r_{-}\}$ by the triangle inequality in $\mathcal Y$.
Also, $r_{\pm}\le \lambda\,\norm{y_+}_{\cb}^{\alpha/2}=\lambda$ by Lemma~\ref{lem:deltareg-basic}.

Fix the anchor $\theta:=x_1=\Theta\in\CP(M_n,M_m),$
and specialize to the max-glue BK metric 
$d_\theta:=\beta^{BK}_{\theta,\lambda,\infty,\alpha}.$ Let
\[
S:=\{x_0,x_4,y_+,y_-\}\subset \CB(M_n,M_m).
\] where  $x_0:=0,$ and $x_4:=4\Theta$. We deliberately exclude the anchor $\theta=x_1$, because including it tends to cone off Rips homology. With this notation as in place, the BK distances on $S$ are:
\begin{equation}\label{eq:bures-matrix2}
d_\theta(x_0,x_4)=2,\quad d_\theta(y_+,y_-)=D,\quad d_\theta(x_0,y_\pm)=\max\{1,r_{\pm}\},\quad d_\theta(x_4,y_\pm)=\max\{1,r_{\pm}\}.
\end{equation}
Indeed, the CP distances are Bures distances, computed in \eqref{eq:bures-matrix}.
The non-CP distance is $d_{\mathrm{reg}}(y_+,y_-)=D$ by definition.
The cross distances follow from $\beta(x_0,\theta)=1$ and 
$\beta(x_4,\theta)=1$.
\end{example}

We are now computing $\VR_t(S,d_\theta)$ exactly, in terms of $r$ and $D$.

\begin{theorem}
\label{thm:bk-rips-loop}
Assume $\max\{1,r_{-}, r_{+}\}\le t<\min\{2,D\}$.
Then the Vietoris--Rips complex $\VR_t(S,d_\theta)$ is the clique complex of the complete bipartite graph $K_{2,2}$,
hence is homotopy equivalent to $S^1$ and has $\mathrm{rank}\,H_1=1$.
\end{theorem}

\begin{proof}
By \eqref{eq:bures-matrix2} the cross distances satisfy
$d_\theta(x_0,y_\pm)=d_\theta(x_4,y_\pm)\le t$,
so all four mixed edges $\{x_i,y_\pm\}$ are present.
The within-side distances satisfy $d_\theta(x_0,x_4)=2>t$ and $d_\theta(y_+,y_-)=D>t$,
so there are no edges inside $\{x_0,x_4\}$ and none inside $\{y_+,y_-\}$.
Thus the 1-skeleton is exactly $K_{2,2}$, and since $\VR_t$ is a flag complex, it is the clique complex of $K_{2,2}$,
which is a $4$-cycle and homotopy equivalent to $S^1$.
\end{proof}

\begin{remark}
The condition $\max\{1,r_{\pm}\}\le t<\min\{2,D\}$ is exactly the condition that all mixed edges exist, controlled only by the radii to the anchor/basepoint,
but no within-side edges exist.
This is precisely the wedge-geometry mechanism behind the appearance of loops in BK Rips complexes.
\end{remark}

\subsection{\v{C}ech complexes}

Let $(Z,d)$ be a metric space and let $S\subseteq Z$.
In the literature, the \v{C}ech complex is most commonly understood as the nerve of a cover by metric balls in an ambient space, typically with $Z\subseteq \mathbb{R}^n$, where intersections are taken in the ambient space; see \cite{EdelsbrunnerHarer, Oudot, deSilvaGhrist}.
In contrast, for a general metric space $(Z,d)$ one may also consider an intrinsic \v{C}ech complex, where balls are taken in $Z$ itself and intersections are required to lie in $Z$.
To our knowledge, this distinction is not systematically emphasized in the literature, but it plays a crucial role in wedge-type metric spaces such as the BK construction.

In the BK wedge setting, this distinction is essential: ambient \v{C}ech complexes can detect the glued basepoint and become contractible at relatively small scales, while intrinsic \v{C}ech complexes retain the combinatorial geometry of the cloud.
We now make these two constructions precise.

\begin{definition}
Let $(Z,d)$ be a metric space and $t\ge 0$.

(i) The intrinsic \v{C}ech complex $\check{C}_t(Z,d)$ is the nerve of the cover of $Z$ by
(closed) metric balls $\{ \overline{B}_Z(z,t)\}_{z\in Z}$.
Equivalently, a finite subset $\sigma\subseteq Z$ is a simplex if and only if $\bigcap_{z\in\sigma} \overline{B}_Z(z,t)\neq\varnothing.$

(ii) The ambient \v{C}ech complex of a cloud $S$ in $Z$ is the abstract simplicial complex $\check{C}^{\mathrm{amb}}_t(S\subseteq Z,d)$
with vertex set $S$ such that a finite subset $\sigma\subseteq S$ is a simplex if and only if $\bigcap_{s\in\sigma} \overline{B}_Z(s,t)\neq\varnothing.$

When the ambient space is clear, we simply write $\check{C}^{\mathrm{amb}}_t(S,d)$.
\end{definition}

\begin{conventions} (i) (convention for finite clouds.)
We reserve $\check{C}_t(Z,d)$ for the intrinsic \v{C}ech complex of the metric space $Z$ itself.
For a finite cloud $S\subseteq Z$, the ambient complex
$\check{C}^{\mathrm{amb}}_t(S\subseteq Z,d)$ allows witnesses anywhere in $Z$,
whereas the intrinsic complex $\check{C}_t(S,d|_S)$ allows only witnesses in $S$.
Hence
\[
\check{C}_t(S,d|_S)\ \subseteq\ \check{C}^{\mathrm{amb}}_t(S\subseteq Z,d).
\]
All basepoint-driven contractibility statements for finite BK clouds below are ambient statements.

(ii) (radius-$t$ convention.) Some authors use the ``$t/2$'' convention (balls of radius $t/2$) so that $\check{C}_t$
approximates the union of $t/2$-balls.
All statements below use the radius-$t$ convention.

For every metric space $(Z,d)$ and every cloud $S\subseteq Z$, one has
\[
\VR_t(S,d|_S)\ \subseteq\ \check{C}^{\mathrm{amb}}_t(S\subseteq Z,d)\ \subseteq\ \VR_{2t}(S,d|_S).
\]
Indeed, the usual proof of $\VR_t\subseteq \check{C}_t\subseteq \VR_{2t}$ works verbatim in the
ambient setting. When $S=Z$, this reduces to the intrinsic inclusion
\[
\VR_t(Z,d)\ \subseteq\ \check{C}_t(Z,d)\ \subseteq\ \VR_{2t}(Z,d).
\]
\end{conventions}
Metric balls in the wedge admit an explicit componentwise description.

\begin{lemma}
\label{prop:ball-description}
Fix $t\ge 0$.

\smallskip\noindent
\textup{(i)} If $x\in\mathcal C$, then $\overline{B}_{\mathcal X}(x,t)\cap\mathcal C
=
\overline{B}_{\mathcal C}(x,t),$
and $\overline{B}_{\mathcal X}(x,t)\cap\mathcal Y
=
\left\{\,y\in\mathcal Y:\ \bigl\|\bigl(r_{\mathcal C}(x),r_{\mathcal Y}(y)\bigr)\bigr\|_{\ell^p}\le t\,\right\}.$

\smallskip\noindent
\textup{(ii)} If $y\in\mathcal Y$, then $\overline{B}_{\mathcal X}(y,t)\cap\mathcal Y
=
\overline{B}_{\mathcal Y}(y,t),$
and
$\overline{B}_{\mathcal X}(y,t)\cap\mathcal C
=
\left\{\,x\in\mathcal C:\ \bigl\|\bigl(r_{\mathcal C}(x),r_{\mathcal Y}(y)\bigr)\bigr\|_{\ell^p}\le t\,\right\}.$
\end{lemma}

\begin{proof}
Inside each component, the metric agrees with the component metric.
Across components, the distance is given by \eqref{eq:cross-distance}, so the cross-description follows directly.
\end{proof}

Mixed \v{C}ech simplices can be characterized in terms of constrained ball intersections within a single component.

\begin{proposition}
\label{prop:cech-mixed-criterion}
Fix $t\ge 0$.
Let $\sigma=\{x_1,\dots,x_m\}\subseteq\mathcal C$ and $\tau=\{y_1,\dots,y_n\}\subseteq\mathcal Y^\circ$
be finite subsets, and set \[A:=\max_{1\le i\le m} r_{\mathcal C}(x_i),\qquad
B:=\max_{1\le j\le n} r_{\mathcal Y}(y_j).\]
Then $\sigma\cup\tau$ is a simplex of $\check{C}_t(\mathcal X,d_\theta)$ if and only if at least one of the
following two conditions holds:

\begin{enumerate}
\item[(C)] There exists $z\in\mathcal C$ such that
$z\in \bigcap_{i=1}^m \overline{B}_{\mathcal C}(x_i,t)$ and 
$\bigl\|\bigl(r_{\mathcal C}(z),B\bigr)\bigr\|_{\ell^p}\le t.$

\item[(Y)] There exists $w\in\mathcal Y$ such that $w\in \bigcap_{j=1}^n \overline{B}_{\mathcal Y}(y_j,t)$ and 
$\bigl\|\bigl(A,r_{\mathcal Y}(w)\bigr)\bigr\|_{\ell^p}\le t.$
\end{enumerate}
\end{proposition}

\begin{proof}
By definition, $\sigma\cup\tau$ is a \v{C}ech simplex iff $\bigcap_{u\in \sigma\cup\tau}\overline{B}_{\mathcal X}(u,t)\neq\varnothing.$ Any point of the wedge $\mathcal X\cong \mathcal C\vee_p\mathcal Y$ is represented either by a point of
$\mathcal C$ or by a point of $\mathcal Y$.
Hence the intersection is nonempty iff it contains a point $z\in\mathcal C$ or a point $w\in\mathcal Y$.

Assume first that there is $z\in\mathcal C$ in the intersection.
Then $z\in\overline{B}_{\mathcal X}(x_i,t)\cap\mathcal C=\overline{B}_{\mathcal C}(x_i,t)$ for all $i$,
and also $d_\theta(z,y_j)\le t$ for all $j$.
Using $d_\theta(z,y_j)=\bigl\|\bigl(r_{\mathcal C}(z),r_{\mathcal Y}(y_j)\bigr)\bigr\|_{\ell^p}$ and monotonicity in the second coordinate, the inequalities $d_\theta(z,y_j)\le t$ for all $j$
are equivalent to $\bigl\|\bigl(r_{\mathcal C}(z),B\bigr)\bigr\|_{\ell^p}\le t.$ This is exactly condition \textup{(C)}.

Conversely, if \textup{(C)} holds, then for each $j$ we have $r_{\mathcal Y}(y_j)\le B$, hence
\[
d_\theta(z,y_j)
=
\bigl\|\bigl(r_{\mathcal C}(z),r_{\mathcal Y}(y_j)\bigr)\bigr\|_{\ell^p}
\le
\bigl\|\bigl(r_{\mathcal C}(z),B\bigr)\bigr\|_{\ell^p}
\le t,
\]
and also $d_\theta(z,x_i)=\beta(z,x_i)\le t$ for all $i$.
Thus $z$ lies in all ambient balls.
The proof for \textup{(Y)} is symmetric, using Lemma~\ref{prop:ball-description} and monotonicity
in the first coordinate.
\end{proof}

The preceding criterion can be reformulated more explicitly using elementary properties of $\ell^p$ norms.

\begin{observation}
Let $1\le p<\infty$, and fix $t\ge 0$.
Then for $a,b\ge 0$, $\bigl\|\bigl(a,b\bigr)\bigr\|_{\ell^p}\le t$ if and only if 
$a\le t\ \text{ and }\ b\le (t^p-a^p)^{1/p}$
and equivalently $\bigl\|\bigl(a,b\bigr)\bigr\|_{\ell^p}\le t$
if and only if 
$b\le t\ \text{ and }\ a\le (t^p-b^p)^{1/p}.$

For $p=\infty$ one has
$\bigl\|\bigl(a,b\bigr)\bigr\|_{\ell^\infty}\le t$ if and only if
$a\le t\ \text{ and }\ b\le t.$

Using these equivalences, condition \textup{(C)} above can be rewritten as follows:

\smallskip\noindent
\textup{(i)} If $1\le p<\infty$, then \textup{(C)} is equivalent to
\[
B\le t
\quad\text{and}\quad
\left(\bigcap_{i=1}^m \overline{B}_{\mathcal C}(x_i,t)\right)
\cap
\overline{B}_{\mathcal C}\!\Bigl(\theta,(t^p-B^p)^{1/p}\Bigr)\neq\varnothing.
\]

\smallskip\noindent
\textup{(ii)} If $p=\infty$, then \textup{(C)} is equivalent to
\[
B\le t
\quad\text{and}\quad
\left(\bigcap_{i=1}^m \overline{B}_{\mathcal C}(x_i,t)\right)
\cap
\overline{B}_{\mathcal C}(\theta,t)\neq\varnothing.
\]

\smallskip\noindent
Similarly, condition \textup{(Y)} is equivalent to the following:

\smallskip\noindent
\textup{(iii)} If $1\le p<\infty$, then \textup{(Y)} is equivalent to
\[
A\le t
\quad\text{and}\quad
\left(\bigcap_{j=1}^n \overline{B}_{\mathcal Y}(y_j,t)\right)
\cap
\overline{B}_{\mathcal Y}\!\Bigl(\ast,(t^p-A^p)^{1/p}\Bigr)\neq\varnothing.
\]

\smallskip\noindent
\textup{(iv)} If $p=\infty$, then \textup{(Y)} is equivalent to
\[
A\le t
\quad\text{and}\quad
\left(\bigcap_{j=1}^n \overline{B}_{\mathcal Y}(y_j,t)\right)
\cap
\overline{B}_{\mathcal Y}(\ast,t)\neq\varnothing.
\]
\end{observation}

We now deduce a simple sufficient condition: vertices with sufficiently small radii automatically form mixed \v{C}ech simplices via the basepoint.

\begin{corollary}
\label{cor:basepoint-sufficient}
Fix $t\ge 0$ and let $\sigma\subseteq\mathcal C$ and $\tau\subseteq\mathcal Y^\circ$ be finite subsets.
If $\max_{x\in\sigma} r_{\mathcal C}(x)\le t$ and
$\max_{y\in\tau} r_{\mathcal Y}(y)\le t,$
then $\sigma\cup\tau$ is a simplex of $\check{C}_t(\mathcal X,d_\theta)$.

Consequently, if $S\subseteq\mathcal X$ is a finite cloud and $\sigma\cup\tau\subseteq S$,
then $\sigma\cup\tau$ is also a simplex of the ambient complex
$\check{C}^{\mathrm{amb}}_t(S\subseteq\mathcal X,d_\theta)$.
\end{corollary}

\begin{proof}
Under the stated assumptions, the glued basepoint $\theta\sim\ast$ lies within distance $t$ of every vertex: $d_\theta(\theta,x)=r_{\mathcal C}(x)\le t$ ($x\in\sigma$),
and $d_\theta(\ast,y)=r_{\mathcal Y}(y)\le t$ ($y\in\tau$).
Hence the common intersection of the ambient balls is nonempty.
The final statement is immediate, since the same witness point $\theta\sim\ast\in\mathcal X$
works for the ambient complex of any cloud containing those vertices.
\end{proof}

\begin{remark}
For Vietoris--Rips complexes, mixed simplices are determined entirely by the maxima $(A,B)$ via
Proposition~\ref{prop:rips-mixed}.
For \v{C}ech, Proposition~\ref{prop:cech-mixed-criterion} shows that one must additionally understand
whether certain intersections of balls inside one component meet a suitable sublevel ball around the
glued basepoint.
Thus \v{C}ech complexes retain more geometric information about ball-intersection patterns inside
$(\mathcal C,\beta)$ and $(\mathcal Y,d_{\mathrm{reg}})$.
\end{remark}

\begin{figure}[t]
\centering
\begin{tikzpicture}[scale=0.95,
    every node/.style={scale=0.92},
    cp/.style={circle,draw,fill=blue!15,inner sep=1.7pt},
    ncp/.style={circle,draw,fill=red!15,inner sep=1.7pt},
    aux/.style={circle,draw,fill=gray!20,inner sep=1.5pt},
    lbl/.style={draw=none}
]

\node[lbl] at (2,2.45) {\small BK Rips at scale $t$};

\node[cp]  (x1L) at (0,1) {$x_1$};
\node[cp]  (x2L) at (0,-1) {$x_2$};
\node[ncp] (y1L) at (4,1) {$y_1$};
\node[ncp] (y2L) at (4,-1) {$y_2$};

\node[lbl] at (0,1.65) {\scriptsize $\mathcal C$};
\node[lbl] at (4,1.65) {\scriptsize $\mathcal Y$};

\draw[very thick] (x1L) -- (y1L);
\draw[very thick] (x1L) -- (y2L);
\draw[very thick] (x2L) -- (y1L);
\draw[very thick] (x2L) -- (y2L);

\draw[dashed] (x1L) -- (x2L);
\draw[dashed] (y1L) -- (y2L);

\node[lbl,align=center] at (2,-2.0)
{\scriptsize No within-side edges, so no triangles occur.\\
 \scriptsize Hence $\VR_t$ is the graph $K_{2,2}\simeq S^1$.};

\node[lbl] at (5.9,0) {$\Longrightarrow$};

\node[lbl] at (10.0,2.45) {\small Ambient BK \v{C}ech at the same scale};

\node[cp]  (x1R) at (7.8,1) {$x_1$};
\node[cp]  (x2R) at (7.8,-1) {$x_2$};
\node[ncp] (y1R) at (11.8,1) {$y_1$};
\node[ncp] (y2R) at (11.8,-1) {$y_2$};

\node[aux] (thR) at (9.8,0) {$\theta$};

\node[lbl] at (7.8,1.65) {\scriptsize $\mathcal C$};
\node[lbl] at (11.8,1.65) {\scriptsize $\mathcal Y$};

\draw[dotted] (thR) -- (x1R);
\draw[dotted] (thR) -- (x2R);
\draw[dotted] (thR) -- (y1R);
\draw[dotted] (thR) -- (y2R);

\fill[gray!12] (x1R.center) -- (y1R.center) -- (x2R.center) -- (y2R.center) -- cycle;

\draw[thick] (x1R) -- (y1R);
\draw[thick] (y1R) -- (x2R);
\draw[thick] (x2R) -- (y2R);
\draw[thick] (y2R) -- (x1R);

\draw[thick] (x1R) -- (x2R);
\draw[thick] (y1R) -- (y2R);

\node[lbl,align=center] at (9.8,-2.0)
{\scriptsize The ambient balls all meet at the hidden witness $\theta$.\\
 \scriptsize Hence $\check C_t^{\mathrm{amb}}$ is the full simplex, so contractible.};

\end{tikzpicture}

\caption{A schematic comparison between BK Vietoris--Rips and ambient BK \v{C}ech complexes at the same scale $t$. 
Left: in the Rips complex, only the mixed edges are present, so the 1-skeleton is $K_{2,2}$ and the complex is a $4$-cycle. 
Right: in the ambient \v{C}ech complex, the glued basepoint $\theta\sim\ast$ may lie in all ambient balls without being a vertex of the cloud; this hidden witness fills in all simplices, so the complex becomes contractible. This illustrates the basepoint-driven cone effect from Corollary~\ref{cor:cech_cone_effect} and the contrast with Corollary~\ref{ex:universal_loop}.}
\label{fig:bk-rips-vs-cech}
\end{figure}
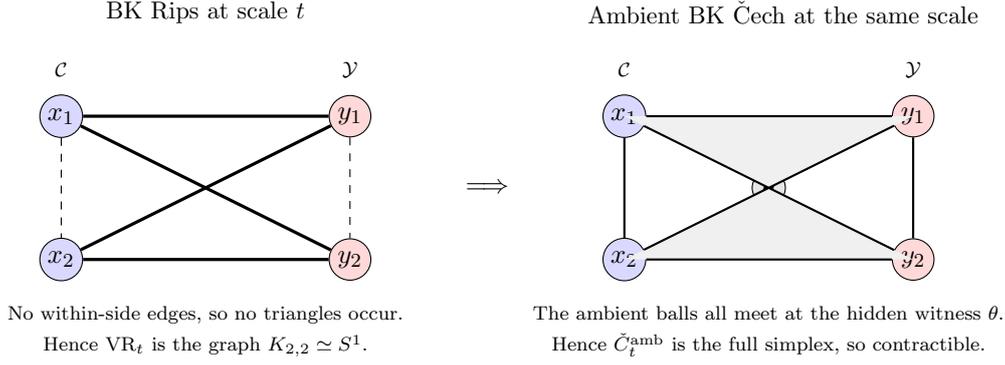

Ambient \v{C}ech complexes exhibit a basepoint-driven cone effect.
\begin{corollary}
\label{cor:cech_cone_effect}
Let $S\subset\mathcal X$ be finite and let $t\ge 0$.
Then the induced subcomplex of the ambient \v{C}ech complex $\check{C}^{\mathrm{amb}}_t(S\subseteq \mathcal X,d_\theta)$
on the vertex set $S_{\mathcal C}^{\le t}\cup S_{\mathcal Y}^{\le t}$ is a full simplex, hence contractible.
In particular, if every vertex of $S$ has radius $\le t$, then
$\check{C}^{\mathrm{amb}}_t(S\subseteq \mathcal X,d_\theta)$ is a simplex and therefore contractible.
\end{corollary}

\begin{proof}
Every finite subset of $S_{\mathcal C}^{\le t}\cup S_{\mathcal Y}^{\le t}$ satisfies the hypothesis of
Corollary~\ref{cor:basepoint-sufficient}, so every such subset spans a simplex in the ambient complex.
Hence the induced subcomplex is the full simplex on that vertex set.
\end{proof}

\begin{remark}
[See Figure 3.]
Comparing Corollary~\ref{ex:universal_loop} with Corollary~\ref{cor:cech_cone_effect}, we observe the following contrast.
In Corollary~\ref{ex:universal_loop}, the Rips complex at scale $t$ is a loop $S^1$.
But if all four vertices have radius $\le t$, then the ambient \v{C}ech complex at the same scale
is a $3$-simplex, hence contractible.
Thus ambient BK \v{C}ech filtrations can collapse homology much earlier than BK Rips filtrations.
\end{remark}
The following example shows that intrinsic and ambient \v{C}ech filtrations behave differently even in simple CP clouds.
\begin{example}[Intrinsic and ambient \v{C}ech filtrations for the CP cloud]\label{prop:cp-cech}
Let $S_{\mathcal C}:=\{x_0,x_1,x_4\}\subset \CP(M_n,M_m)$ be as in Example~\ref{ex:go1}.

\smallskip\noindent
\textup{(a) Intrinsic \v{C}ech on the finite metric space $S_{\mathcal C}$.}
One has
\[
\check{C}_t(S_{\mathcal C},\beta|_{S_{\mathcal C}})
=
\begin{cases}
\text{three isolated vertices}, & 0\le t<1,\\[1mm]
\text{the full $2$-simplex on $\{x_0,x_1,x_4\}$}, & t\ge 1.
\end{cases}
\]
Indeed, if $0\le t<1$, then each intrinsic ball in $S_{\mathcal C}$ is a singleton, because all nonzero
pairwise distances are at least $1$.
If $t\ge 1$, then $x_1\in S_{\mathcal C}$ lies in all three intrinsic balls, since
\[
\beta(x_1,x_0)=1,
\qquad
\beta(x_1,x_4)=1.
\]
Hence the three balls have nonempty common intersection.

\smallskip\noindent
\textup{(b) Ambient \v{C}ech in the Bures ray.}
Let
\[
R:=\{\Theta_c:c\ge 0\}\subset \CP(M_n,M_m).
\]
Then
\[
\check{C}^{\mathrm{amb}}_t(S_{\mathcal C}\subseteq R,\beta)
=
\begin{cases}
\text{three isolated vertices}, & 0\le t<\tfrac12,\\[1mm]
\text{the path graph }x_0\!-\!x_1\!-\!x_4, & \tfrac12\le t<1,\\[1mm]
\text{the full $2$-simplex on $\{x_0,x_1,x_4\}$}, & t\ge 1.
\end{cases}
\]
Indeed, by Lemma~\ref{lem:bures-ray}, the ray $R$ is isometric to $[0,\infty)$ via $x_c=\Theta_c\mapsto \sqrt c$,
so the three points $x_0,x_1,x_4$ correspond to $0,1,2$.
Two radius-$t$ ambient balls intersect iff the centers are at distance at most $2t$.
Hence the pairs $(x_0,x_1)$ and $(x_1,x_4)$ intersect for $t\ge \tfrac12$, while $(x_0,x_4)$ intersects for $t\ge 1$.
For $t\ge 1$, the point $x_1$ lies in all three ambient balls, so the nerve contains the filled triangle.
\end{example}

\usetikzlibrary{positioning,calc}

\begin{figure}[t]
\centering
\begin{tikzpicture}[scale=0.85,
    every node/.style={scale=0.9},
    pt/.style={circle,draw,fill=blue!15,inner sep=1.5pt},
    lbl/.style={draw=none}
]

\node[lbl] at (1.6,2.3) {\small Intrinsic \v{C}ech on $S_{\mathcal C}$};

\node[pt] (x0L) at (0,0) {$x_0$};
\node[pt] (x1L) at (1.6,0) {$x_1$};
\node[pt] (x4L) at (3.2,0) {$x_4$};

\draw[dashed,gray] (-0.3,0) -- (3.5,0);

\node[lbl,align=center] at (1.6,-1.6)
{\scriptsize $t<1$: isolated vertices\\
 \scriptsize $t\ge1$: full simplex};

\node[lbl] at (4.8,0) {$\Longrightarrow$};

\node[lbl] at (8.7,2.3) {\small Ambient \v{C}ech in $R\cong[0,\infty)$};

\draw[->] (6.5,0) -- (11,0);

\node[pt] (x0R) at (7,0) {$x_0$};
\node[pt] (x1R) at (8.7,0) {$x_1$};
\node[pt] (x4R) at (10.4,0) {$x_4$};

\draw[thick] (8,0.45) -- (9.4,0.45);
\draw[thick] (6.6,0.95) -- (8.1,0.95);
\draw[thick] (9.3,0.95) -- (10.8,0.95);

\node[lbl,align=center] at (8.7,-1.6)
{\scriptsize $t<\frac12$: isolated\\
 \scriptsize $\frac12\le t<1$: path\\
 \scriptsize $t\ge1$: simplex};

\end{tikzpicture}

\caption{Example~\ref{prop:cp-cech}. 
Intrinsic and ambient \v{C}ech filtrations behave differently even for a simple three-point CP cloud. 
The intrinsic complex jumps directly from discrete to a simplex, while the ambient complex exhibits an intermediate path stage due to intersections occurring in the ambient Bures ray.}
\label{fig:cp-cech-intrinsic-ambient}
\end{figure}
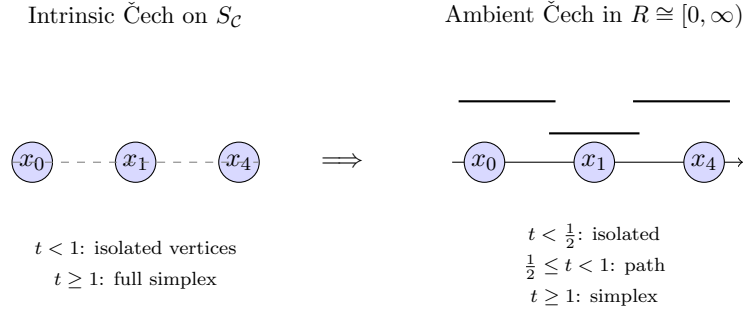

\begin{remark}
Thus the intermediate path stage occurs for the ambient \v{C}ech filtration, but not for the
intrinsic \v{C}ech filtration on the three-point metric space itself.
More precisely, the intrinsic \v{C}ech complex becomes contractible at $t=1$, when it jumps
directly to the full $2$-simplex. By contrast, the ambient \v{C}ech complex is already
contractible for $\tfrac12 \le t < 1$, since in that range it is the path graph
$x_0\!-\!x_1\!-\!x_4$, and it becomes the full $2$-simplex at $t=1$.
Likewise, $\VR_t(S_{\mathcal C},\beta)$ becomes contractible at $t=1$, where it is again the
path graph $x_0\!-\!x_1\!-\!x_4$, and becomes the full $2$-simplex only at $t=2$.
\end{remark}

Mixed BK clouds yield contractible ambient \v{C}ech complexes at moderate scales.
\begin{example}[Ambient \v{C}ech contractibility on the mixed cloud]
\label{thm:bk-cech-contractible}
Let $S:=\{x_0,x_4,y_+,y_-\}\subset \CB(M_n,M_m)$ be as in Example~\ref{ex:go2}.
If $t\ge \max\{1,r_-,r_+\},$
then the ambient \v{C}ech complex
$\check{C}^{\mathrm{amb}}_t(S\subseteq \mathcal X,d_\theta)$
is the full simplex on the four vertices $S$, hence contractible.

Indeed, since $\theta=x_1$ is the anchor, one has
\[
d_\theta(x_0,\theta)=\beta(x_0,\theta)=1,
\qquad
d_\theta(x_4,\theta)=\beta(x_4,\theta)=1.
\]
Also, $d_\theta(y_+,\theta)=r_+,$ and
$d_\theta(y_-,\theta)=r_-.$
Therefore $\theta\in \overline{B}_{\mathcal X}(s,t)$ for every $s\in S$.
So the four ambient balls have nonempty common intersection, and the ambient \v{C}ech complex is the full simplex.
\end{example}

\begin{remark}
Theorem~\ref{thm:bk-rips-loop} and Example~\ref{thm:bk-cech-contractible} show a typical BK phenomenon:
the Rips complex can exhibit a loop created by the wedge-join geometry, while the ambient \v{C}ech complex
can already be contractible at the same scale because the glued basepoint lies in all ambient balls,
even though it is not itself a vertex of the cloud.
\end{remark}


\end{document}